\documentclass[a4paper,10pt]{article}
\usepackage{amsfonts}
\usepackage{amsthm}
\usepackage{amssymb}
\usepackage{fullpage}
\usepackage{enumerate}
\usepackage{pgf,tikz,tikz-cd}
\usepackage{graphicx}
\usepackage{multirow}
\usepackage{setspace}
\usepackage{listings}
\usepackage{caption}
\usepackage{subcaption}
\usepackage{epstopdf}
\usepackage{mathrsfs}
\usepackage{natbib}
\usepackage{hyperref,url}
\usepackage{algorithm}
\usepackage{ulem}
\usepackage[noend]{algpseudocode}
\usepackage{MnSymbol}
\usepackage{comment}
\usepackage{array}

\usepackage{authblk}

\newcommand{\CC}{\mathbb{C}}
\newcommand{\RR}{\mathbb{R}}
\newcommand{\NN}{\mathbb{N}}

\newcommand{\ZZ}{\mathbb{Z}}

\newcommand{\EE}{\mathbb{E}}

\newcommand{\PP}{\mathbb{P}}

\newcommand{\dd}{\,\mathrm{d}}

\newcommand{\iid}{\overset{iid}{\sim}}

\newtheorem{thm}{Theorem}

\newtheorem{Cor}[thm]{Corollary}

\newtheorem{prop}[thm]{Proposition}
\newtheorem{eg}[thm]{Example}

\newtheorem{Lem}[thm]{Lemma}

\newtheorem{rem}[thm]{Remark}
\newcommand{\bds}{\begin{displaystyle}}
\newcommand{\eds}{\end{displaystyle}}
\newcommand{\bpm}{\begin{pmatrix}}
\newcommand{\epm}{\end{pmatrix}}
\newcommand{\bvm}{\begin{vmatrix}}
\newcommand{\evm}{\end{vmatrix}}

\makeatletter
\newcommand{\chaptersubheading}[1]{%
  {\parindent0pt\vspace*{-25pt}%
  \linespread{1.1}\large\scshape#1%
  \par\nobreak\vspace*{35pt}}
  \@afterheading%
}
\makeatother

\newif\ifhideproofs

\ifhideproofs
\usepackage{environ}
\NewEnviron{hide}{}

\fi

\usepackage{xcolor}
\def\cblu{\color{blue}}

\numberwithin{equation}{section}
\numberwithin{thm}{section}

\title{Symmetry: a General Structure in Nonparametric Regression}
\author[1]{Louis G. Christie}
\author[1]{John A. D. Aston} 
\affil[1]{Statistical Laboratory, University of Cambridge}

\begin{document}

\maketitle

\begin{abstract}
In this paper we present the framework of symmetry in nonparametric regression. This generalises the framework of covariate sparsity, where the regression function $f : [0,1]^d \rightarrow \RR$ depends only on at most $s < d$ of the covariates, which is a special case of translation symmetry with linear orbits. In general this extends to other types of functions that capture lower dimensional behavior even when these structures are non-linear. We show both that known symmetries of regression functions can be exploited to give similarly faster rates, and that unknown symmetries with Lipschitz actions can be estimated sufficiently quickly to obtain the same rates. This is done by explicit constructions of partial symmetrisation operators that are then applied to usual estimators, and with a two step $M$-estimator of the maximal symmetry of the regression function. We also demonstrate the finite sample performance of these estimators on synthetic data. 
\end{abstract}

\section{Introduction}

In nonparametric regression, the dimension of the covariate space has a large impact on the ability to estimate the regression function. For example, it is well known that the minimax rate over the class $\mathcal{F}(L, \beta) \subseteq L^2( [0,1]^d )$ of $(L, \beta)$-H\"{o}lder functions in  (defined in section \ref{ssec:stat_problem}) is:
\begin{equation}    
    \label{eq:curse_of_dim}
    \inf_{\hat{f}} \sup_{f \in \mathcal{F}( L, \beta )} \EE( \| \hat{f} - f \|_2^2 ) \geq C n^{-\frac{2 \beta }{2\beta + d} } 
\end{equation}
for some positive constant $C$, where the infimum is taken over all estimators $\hat{f}$ using data $\{ (X_i, f(X_i) + \epsilon_i ) \}_{i =1 }^n$ with $\epsilon \iid N(0, \sigma^2)$ independent of $X_i \iid U( [0,1]^d )$. This rate is achieved for many estimators in many problem contexts, for example local polynomial regressors \citep{stone1982optimal}, projection estimators \citep{rice1984bandwidth}, or well designed neural networks (up to log factors) \citep{schmidt2020nonparametric}. The rate means that number of samples to bring the expected error below $\varepsilon > 0$ is exponential in the dimension; a problem known as the \textit{curse of dimensionality}. \\
\\
To overcome this we usually either make assumptions on the structure of $f$, e.g. that it depends only on $s < d$ of the covariates (i.e., $f$ is \textit{$s$-sparse}), or on the distribution of the data, e.g., that $\mu$ is concentrated around a $k < d$ dimensional sub-manifold of the covariate space. When true these assumptions allow for much faster estimation, as is done by the RODEO in \cite{lafferty2008rodeo}. However, the sparsity and manifold hypotheses are very restrictive. For example, the regression function can depend on a non-linear function of the covariates, such as the height of a ripple depending on $\| X \|_2$ and thus all $d$ of the covariates. When these assumptions do not hold they either introduce asymptotic bias or revert to the baseline rate governed by \ref{eq:curse_of_dim}. Thus there is significant benefit in generalising these structural assumptions to achieve faster rates in more problem contexts. \\
\\
In this paper we present a more general framework: \textit{symmetry}. Symmetries are ubiquitous across statistics in fields ranging from time series obeying seasonality to rotation invariant protein volume estimation \citep{jiang2017atomic}. In regression, symmetry is taking hold across literature in neural networks \citep{kondor2018generalization, bronstein2021geometric}, and generalises several structures that are well studied such as covariate sparsity and multi-index models (see  examples \ref{eg:cov_sparsity} and \ref{eg:multi_ind_models} below). Lie groups were initially conceived to study the properties of differential equations \citep{hawkins2012emergence}; now in statistics we have an opportunity to use these mathematical tools to better understand problems in regression and beyond. \\
\\
The key idea is that if $G$ is a set of differentiable bijections $\phi: \mathcal{X} \rightarrow \mathcal{X}$ of the $d$-dimensional covariate space $\mathcal{X}$,  and the regression function $f : \mathcal{X} \rightarrow \RR$ is $G$-invariant, i.e., it obeys the invariance rule $f(x) = f( \phi(  x) )$ for all $\phi \in G$ and $x \in \mathcal{X}$, then we can use estimates of each $f( \phi(x) )$ for $\phi \in G$ to construct a lower variance estimator of $f(x)$ without introducing bias. We call the set $\{ \phi(x) : \phi \in G \}$ the \textbf{orbit} of $x$ under $G$, written $[x]_G$,  and when this set is smooth and has positive dimension\footnote{the principle orbit theorem ensures that almost all orbits $[x]_G$ have the same dimension, see section \ref{ssec:group_theory}.} $d_G$ we might hope that the new estimate has a worst case risk decay with $d$ replaced by $d - d_G$. \\  
\\
To construct an estimator with the rate corresponding to the lower dimensional space, we apply a \textit{symmetrisation operator} to any existing estimator $f_n$ satisfying some optimality conditions akin to local polynomial regressors. For example, if $\phi_i \iid U( G )$ then we take our estimate of $f(x)$ as $\tilde{f}_n = m^{-1} \sum_{ i =1}^m f_n (\phi_i (x) )$. In an ideal world, where each estimate is guaranteed to be independent, this would cut the variance by a factor of $m$. When $f$ is $G$-invariant, it does so without introducing bias. This is the situation in figure \ref{fig:sym_eg}. The first results of this paper show that this is true even when dependence considerations are taken into account, and quantifies the asymptotic bias introduced when $f$ is not $G$-invariant. This allows us to re-optimise the bias-variance trade-off and achieve faster convergence to $f$. \\ 

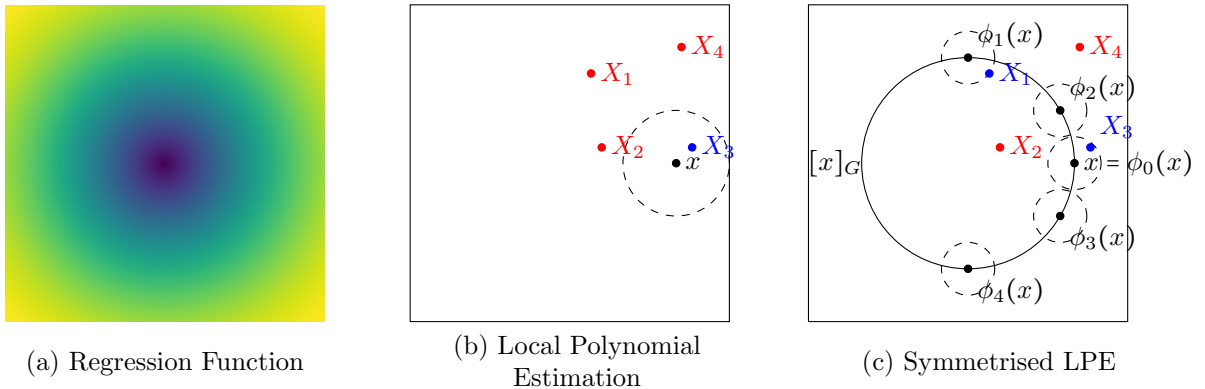
\begin{figure}[h]
    \centering
    \begin{tabular}{m{5cm} m{5cm} m{5cm}}
    \centering
    \includegraphics[scale = 0.45]{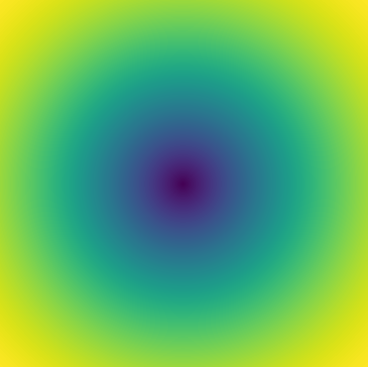}
    & 
    \centering   
    \begin{tikzpicture}[scale = 0.7]
        \draw (-3,-3) -- (3,-3) -- (3,3) -- (-3,3) -- (-3,-3);
        
        \draw[fill] (2,0) circle (0.07cm) node[right]{$x$};
        \draw[dashed] (2,0) circle (1);
        \draw[fill, blue] (2.3,0.3) circle (0.07cm) node[right] {$X_3$};
        \draw[fill, red] (0.4, 1.7) circle (0.07cm) node[right] {$X_1$};
        \draw[fill, red] (0.6, 0.3) circle (0.07cm) node[right] {$X_2$};
        \draw[fill, red] (2.1, 2.2) circle (0.07cm) node[right] {$X_4$};        
    \end{tikzpicture} 
    &
    \centering
    \begin{tikzpicture}[scale = 0.7]
        \draw (0,0) circle (2);

        \draw (-3,-3) -- (3,-3) -- (3,3) -- (-3,3) -- (-3,-3);

        \draw[dashed] (0,2) circle (0.5);
        \draw[dashed] (1.732,1) circle (0.5);
        \draw[dashed] (2,0) circle (0.5);
        \draw[dashed] (1.732,-1) circle (0.5);
        \draw[dashed] (0,-2) circle (0.5);

        \draw[fill] (0,2) circle (0.07cm) node[above right]{$\phi_1( x )$};
        \draw[fill] (1.732,1) circle (0.07cm) node[above right]{$\phi_2( x )$};
        \draw[fill] (2,0) circle (0.07cm) node[right]{$x = \phi_0(x)$};
        \draw[fill] (1.732,-1) circle (0.07cm) node[below right]{$\phi_3( x )$};
        \draw[fill] (0,-2) circle (0.07cm) node[below right]{$\phi_4( x )$};

        \node at (-2.5,0) {$[x]_G$};
        
        \draw[fill, blue] (2.3,0.3) circle (0.07cm) node[above right] {$X_3$};
        \draw[fill, blue] (0.4, 1.7) circle (0.07cm) node[right] {$X_1$};
        \draw[fill, red] (0.6, 0.3) circle (0.07cm) node[right] {$X_2$};
        \draw[fill, red] (2.1, 2.2) circle (0.07cm) node[right] {$X_4$};          
    \end{tikzpicture} 
    \tabularnewline
    \centering (a) Regression Function & 
    \centering (b) Local Polynomial Estimation &
    \centering (c) Symmetrised LPE 
    \end{tabular}

    \caption{     \label{fig:sym_eg} Sub-figure (a) shows a heatmap of a regression function $f : \RR^2  \rightarrow \RR$. Sub-figure (b) depicts the dashed closed ball of radius $h$ around the point $x \in \mathcal{X}$, the data inside which are used to estimate $f(x)$ with a local polynomial estimator $f_n$ with a kernel with support contained in this ball. In this case it is $X_3$ but not $X_1$, $X_2$, nor $X_4$. Sub-figure (c) depicts our estimator which symmetrises with respect to $G = \{ \text{Rotations around } 0 \}$, using multiple balls of smaller radius spaced around the orbit $[x]_G$. In this case the data at both $X_1$ and $X_3$ is used to estimate $f(x)$ with $5^{-1} \sum f_n( \phi_i (x ) )$, but not $X_2$ nor $X_4$.  }

\end{figure}

However, if the symmetry $G$ is not known, we must instead seek to estimate it from the data. Specifically, if $G$ is a large set of possible transformations, then can we estimate the maximal subset $G_0 \subseteq G$ for which $f$ is $G_0$-invariant? If $G_0$ can be estimated sufficiently quickly, then an estimate that uses this symmetry could adapt to the symmetry of $f$, and provide rates of the form:
\begin{equation}
    \label{eq:adapt_rates}
    \sup_{f \in \mathcal{F}(L, \beta)} \EE( \| f_{n, \hat{G}} - f \|_2^2 ) n^{\frac{2 \beta}{2 \beta + d - d_{G_0} }} ) \leq C
\end{equation}
where $f_{n, \hat{G}}$ is an estimator symmetrised with the estimated transformations. This situation is the focus of this paper: we give a (two-step) $M$-estimator\footnote{In the sense of \cite{van2000asymptotic}, chapter 5.} of $G_0$ and show that using this estimated symmetry can achieve adaptive rates of the form \ref{eq:adapt_rates}. \\
\\
This is done through understanding the set $G$ as a \textit{group}, a mathematical object used to describe symmetries. We can do this because if $f$ is $G$-invariant in the sense above then $f$ is also invariant to the sets of inverses $\{ \phi^{-1} : \phi \in G \}$ and compositions $\{ \phi \circ \psi : \phi, \psi \in G \}$. This description is used to reduce the search space from the set of subsets of $G$ to the much smaller set of closed, connected subgroups of $G$, and to keep the language in this paper consistent with the machine learning and mathematical literature. We have aimed to keep the level of group theory background required to a minimum in this paper, introducing any concepts needed. Most of the proofs rely only on elementary probability theory and linear analysis, and on the geometric properties of the orbits. We now demonstrate how this generalises many usual structures in nonparametric regression, we now give several examples. 

\begin{eg}[Covariate Sparsity]
    \label{eg:cov_sparsity}
    Suppose that $\mathcal{X} = \RR^p$ and let $\mathcal{G}$ be the set of translations $\{ \phi_g : x \mapsto x + g \}$. Suppose that the regression function $f : \mathcal{X} \rightarrow \RR$ is $s$-sparse, i.e., there are $(p-s)$ coordinate positions $i$ that $f$ does not depend on, so $f( x + a e_i ) = f(x)$ for all $a \in \RR$ and standard basis vectors $e_i$ for these positions $i$. This is precisely the definition of invariance to the subgroups of $\mathcal{G}$ given by the $\RR$-span of $\phi_{e_i}$. 
\end{eg}

\begin{eg}[Multi-Index Models]
    \label{eg:multi_ind_models}
    In the context of the previous example, there are of course many other subgroups of $\mathcal{G}$ that $f$ could be invariant to, for example translations by the hyperplane orthogonal to the vector $\mathbf{1}_d = (1, \dots, 1)$. A function invariant to this subgroup must be of the form $f(x_1, \dots, x_d) = f( x_1 + \cdots + x_d )$, as the function's value at off diagonal inputs are completely determined by value at the projection to the line through $\mathbf{1}_d$. Thus functions that are invariant to such subgroups are examples of the multi-index structure $f(x) = h( Tx )$ for a orthogonal projection $T : \RR^d \rightarrow \RR^m$ as described in \cite{hristache2001structure}, in this case with $T = \mathbf{1}_d^T$.
\end{eg}

\begin{eg}[Compact Sparsity]
    \label{eg:compact_sparsity}
    The previous example is not an example of a \textbf{compact} group action, but is very related to this example where we consider the compact domain $[0,1]^d$ rather than $\RR^d$. Now consider the action of the $d$-torus $\mathbb{T}^d= \{ \phi_g : g \in [0,1)^d \}$ given by translations modulo the integer grid:
    \begin{equation}
        \phi_g(x) = ( g_1 + x_1 - \lfloor g_1 + x_1 \rfloor, \dots, g_d + x_d - \lfloor g_d + x_d \rfloor ) 
    \end{equation}
    for all $x \in [0, 1)^d$, and $\phi_g(x) = x$ for all other $x$. In this case continuous functions that are invariant to the span of $\phi_{e_i}$ are also sparse, in the sense that they do not depend on the variable in the $i^{th}$ position. The only difference to the previous examples is now the group $\mathbb{T}^d$ is compact (note that the topology on this set identifies the boundaries, so it is homeomorphic to a product of two circles), which can make the search for the maximal subgroup computationally easier.
\end{eg}

\begin{eg}[3D Rotations]
\label{eg:3D_rots}
Suppose that $\mathcal{X} = \overline{ B_{\RR^3}(0, 1) }$, the closed unit ball in $\RR^3$. Let $\mathcal{G} = SO(3) = \{ \text{rotations around 0} \, \}$. Any function $f$ that is invariant to $\mathcal{G}$ depends only on $\| x \|_2$, such a function is not covariate sparse but does only depend on a lower dimensional projection---a non-linear one in this case. If $f$ is invariant to the subgroup of rotations around the $x$ axis, denoted $S^1_{x}$, then $f$ depends only on $x_1$ and $\sqrt{x^2_2 + x^2_3}$, a two dimensional non-linear projection. This also generalises to higher dimensions, where the groups $SO(d)$ refer to rotations around the origin in $d$-dimensions. 
\end{eg}

\subsection{Related Work}

\subsubsection{Using Symmetries in Regression}
The machine learning literature has recently moved to explicitly consider the framework of symmetry. In particular, convolutional neural networks are know to be prime examples of this, utilising translation invariance of images \citep{kondor2018generalization}. A good overview of symmetry-based methods is found in \cite{bronstein2021geometric}, where translation invariant convolutional neural networks or permutation invariant graph neural networks are discussed, along several other example. The main methods for incorporating symmetry in this literature, are:
\begin{enumerate}[1)]
    \item (\textit{Data augmentation}) constructing a new data set of points $( g_{ij} \cdot X_i, Y_i)$ with $j \in \{ 1, \dots, J_i \}$ transformations around the group and estimating using this data set;
    \item (\textit{Data projection}) projecting the covariates $X_i$ to the quotient space $\mathcal{X} / G$ and doing regression of $Y_i$ on the covariates $[X_i]_G$;
    \item (\textit{Feature averaging}) using a symmetrisation operator to average the estimated function $f_n$ over the orbits; and
    \item (\textit{Kernel symmetrisation}) replacing a kernel $k$ with a $G$-invariant kernel $k_G$, in for example kernel ridge regression or local polynomial estimation. 
\end{enumerate}
Theoretical work to analyse these tools has only recently been undertaken. \cite{lyle2020benefits} showed that feature averaged estimators are almost surely better performing, a consequence of the fact that $S_G$ is an orthogonal projection (under our conditions on the group action and when $\mu$ is $G$-invariant) and so $\| f - S_G f_n \|_2 \leq \| S_G \| \| f - f_n \|_2$. \cite{elesedy2021provably} later strengthen this for the case of Kernel Ridge regression to show that the performance improvement is strictly positive, but their bound decays as $O(n^{-1})$ so does not guarantee dimension reduction. \cite{bietti2021sample} showed that if $G$ is finite then the risk of a kernel symmetrised kernel ridge regressor converges to $1 / |G|$ of the risk of the original estimator, with a new choice of hyper-parameters. \cite{mei2021learning} examined the performance of kernel symmetrisation on kernel ridge regressors in under- and over-parametrised regimes. \cite{huang2022quantifying} have shown that in some cases data augmentation can introduce unintended variance depending on the estimator and the data distribution. The first works on dimension reduction in this case are \cite{tahmasebi2023exact, tahmasebi2023sample}, which examine how known symmetries affect the estimation of regression functions using symmetrised kernel ridge regressors and the estimation of distributions with projection estimators respectively.

\subsubsection{Estimating Symmetries}
Efforts to estimate symmetries are considerably more limited. Some methods include \citet{cubuk2019autoaugment, lim2019fast, benton2020learning}. These methods focus only on subsets of the same dimension of the group. Moreover, little if any attention has been given to the statistical properties of these methods. \citet{christie2023estimating} provides a framework for statistical inference of the maximal subgroup of a regression function, but rates of convergence are not known for this estimator. It is not known whether it is possible to estimate the true maximal symmetry $G$ fast enough to achieve the desired dimension reduction. Other works have consider the idea of testing whether a specific symmetry applies to distributions \citep{garcia2020optimal, huang2023multivariate, chiu2023hypothesis} (including conditional distributions, which is nearly the same situation as the regression problem) but not the problem of estimating the largest one.

\subsection{Main Contributions}

We now state the results in this paper in technical language. All terms with definitions contained in section \ref{sec:background} are italicised here. Suppose that $\mathcal{X}$ is a closed subset of $\RR^d$ with non-trivial Lebesgue volume, or any $d$-dimensional \textit{orientable Riemannian manifold}, and let $\mathcal{G}$ be a locally compact \textit{Lie group acting smoothly, faithfully}, and \textit{properly} on $\mathcal{X}$. Let $\mathcal{F}(L, \beta)$ be the \textit{H\"{o}lder class} of functions on $\mathcal{X}$. Suppose that we have collected independent and identically distributed (iid) data $\mathcal{D} = \{ (X_i, Y_i) \}_{i = 1}^n$ with $Y_i = f(X_i) + \epsilon_i$ where $\EE( \epsilon_i ) = 0$ and $\mathrm{Var}( \epsilon_i ) = \sigma^2$, and with $X_i$ and $\epsilon_i$ independent. Suppose that $f_n$ is a rate optimal estimator over the hypothesis class $\mathcal{F}(L, \beta)$, such as a local polynomial regressor of degree $\ell = \lceil \beta \rceil - 1$. Let $S_\rho : L^2(\mathcal{X}) \rightarrow L^2( \mathcal{X})$ be the \textit{partial symmetrisation operator} given by the orbit average $S_\rho f = \EE( f( g \cdot x) )$ when $g \sim \rho$ for any distribution $\rho$ on the group $\mathcal{G}$. For each $f \in \mathcal{F}$, let $G_{\max}(f, \mathcal{G})$ be the unique \textit{maximal invariant subgroup} of $f$. In this paper:

\begin{enumerate}[1)]
\item (Sections \ref{ssec:par_sym_ests} and \ref{ssec:point_wise_risk}) we construct \textbf{Partially Symmetrised Estimators} $S_{\rho_G} f_n$ using specific deterministic discrete distributions $\rho(G)$ on each closed subgroup $G \leq \mathcal{G}$ for which we obtain pointwise bounds for all $x \in \mathcal{X}$: 
\begin{equation}
    \label{eq:G_fast_rates}
    \sup_{f \in \mathcal{F}( L ,\beta) } \EE\big( ( S_{\rho(G)} f_n (x) - S_{\rho(G)} f(x) )^2 \big) \leq C_{G,x} n^{- \frac{2\beta}{2 \beta + d - d^G } }
\end{equation}
for positive constants $C_{G,x}$ that do not depend on $f$ or $n$, and where $d^G$ is the dimension of a principle orbit of $G$; 
\item (Section \ref{ssec:integrated_risk}) we show that when $\Omega \subseteq \mathcal{X}$ is compact and $\mu$-measurable we also obtain the integrated risk bound 
\begin{equation}
    \sup_{f \in \mathcal{F}( L ,\beta) } \EE\big( (S_{\rho(G)} f_n (X) - S_{\rho(G)} f(X) )^2 \mid X \in \Omega \big) \leq C_{G, \Omega} n^{- \frac{2\beta}{2 \beta + d - d^G } }
\end{equation}
where $C_{G, \Omega} = \EE( C_{G, X} \mid X \in \Omega ) < \infty$ for $X \sim \mu$;
\item (Section \ref{sssec:compact_groups_comps}) we show that equation (\ref{eq:G_fast_rates}) holds when $\rho(G)$ is replaced by the uniform distribution $U(G)$ for compact groups $G$, and with a the Monte-Carlo estimate with $\{ g_i \}_{i = 1}^n$ sampled uniformly from the group, when $\mu$ is $G$-invariant. This means in practice we do not need the explicit distributions $\rho_G$ for these groups. As a special case, when $f$ is $G$-invariant this recovers the rates of \cite{tahmasebi2023exact} as a corollary, though in a slightly different context; 
\item (Section \ref{ssec:subgroup_space}) we develop a metric structure on the set of closed subgroups $K(\mathcal{G}) = \{ \overline{G} : G \leq \mathcal{G} \}$, and show that the asymptotic bias of the partially symmetrised estimators $S_{\rho(G)} f_n$ is locally H\"{o}lder continuous in this metric, when the distributions $\rho(G)$ each has compact support contained in $G$; 
\item (Section \ref{ssec:global_sym_rates}) we show that when $\Omega \subseteq \mathcal{X}$ is compact closure of an open subset we can compute an \textbf{Error Minimising Symmetry} $\hat{G}$, an estimator of $G_{\max}(f, \mathcal{G})$, and that the \textbf{Best Symmetric Estimator} $S_{ \rho(\hat{G})} f_n$ achieves the adaption property:
\begin{equation}
   \sup_{f \in \mathcal{F}( L ,\beta) } \EE\big( (S_{ \rho(\hat{G})} f_n (X) - f (X) )^2 \mid X \in \Omega \big) n^{ \frac{2\beta}{2 \beta + d - d^{G_{\max}(f, \mathcal{G}) } } } \leq \sup_{G \in K(\mathcal{G} )} 2 C_{G, \Omega} + 2 + \tfrac{ 2\beta (2 C_{G, \Omega} + 4 L ) }{\mu( \Omega) ( 2 \beta + d ) } 
\end{equation}
achieving the goal described in equation (\ref{eq:adapt_rates}). This requires that the action of $\mathcal{G}$ is Lipschitz with respect to the chosen distance on the group. 
\end{enumerate}
Each result here assumes mild and standard conditions on the base estimator $f_n$ and on the group action, which are given explicitly in the assumption sets (E) and (Q) later. The key assumptions of the problem are that the distribution $\mu$ is not concentrated around the singularities of the action, that $\beta$ is at least $\dim [x]_\mathcal{G} / 2$, and that we have a baseline estimator $f_n$ that is sufficiently local, in the same way local polynomial estimators using kernels with compact support are, and minimax optimal for the class $\mathcal{F}( L, \beta) $.

\subsection{Other Notation}

In a metric space $(\mathcal{X}, d_\mathcal{X})$ we use $B_\mathcal{X}(x, h)$ to refer to the open ball of radius $h$, i.e., $B_\mathcal{X}( x, h) = \{ y \in \mathcal{X} : d_{\mathcal{X}}(x,y) < h \}$. We use $\lceil x \rceil$ to mean the smallest integer larger than $x$, i.e., the ceiling function, and $\lfloor x \rfloor$ to be the largest integer not more than $x$, i.e., the floor function. We use $\mathbf{1}_{x \in A}$ for the indicator function on the set $A$.  For any set $A$ we use $|A|$ to refer to it's cardinality. When $x$ is a vector we use $\| x \|_2 = \sqrt{ x_1^2 + \cdots + x_d^2 }$ as its 2-norm, and similarly for functions in $L^p$-spaces we use $\| f \|_p = ( \int f^p )^{1/p}$. We use denote the topological closure of $A$ by $\overline{A}$. We use $\overset{D}{=}$ to mean equality in distribution. We often use the notation $\phi_n^{\beta, d} = n^{- \frac{2\beta}{2\beta + d } }$ to make the rates later more notationally compact.
\section{Background}
\label{sec:background}

\subsection{The Statistical Problem}
\label{ssec:stat_problem}
 Let $\mathcal{X}$ be the closure of any open subset of $\RR^d$, or any other orientable smooth $d$-dimensional Riemannian manifold (defined in section \ref{ssec:diff_geo}) such as the unit sphere $S^d = \{ x \in \RR^{d + 1} : \|x\|_2 = 1\}$. Let $\mu$ be any Borelian probability measure (with respect to the topology on $\mathcal{X}$) that has $\mathrm{supp}(\mu) = \mathcal{X}$. Suppose we collect data $\mathcal{D} = \{ (X_i, Y_i) \}_{i = 1}^n \subseteq \mathcal{X} \times \RR$, with
 \begin{equation}
     Y_i = f(X_i) + \epsilon_i
 \end{equation}
 for some $f \in L^2(\mathcal{X}, \mu)$ and independent and identically distributed (iid) mean $0$ noise variables $\epsilon_i$ that are independent of $\mathcal{D}_X = \{ X_i \}_{i = 1}^n$ and have finite variance. Our goal will always be to estimate $f$ from within a hypothesis class $\mathcal{F} \subseteq L^2(\mathcal{X}, \mu)$.  \\
 \\
 We will primarily consider the \textbf{random design} context where $X_i \iid \mu$, but the results in this paper also hold for fixed design regimes if the $X_i$ are picked such that $B_\mathcal{X}(x,h)$ contains at least one $X_i$ for all $x$ for sufficiently large $n$ and $| B_\mathcal{X}(x,h) \cap \mathcal{D}_X | / n \mu( B_\mathcal{X}(x,h) ) \rightarrow 1$. This is always possible with the choices of $h = a n^{-\frac{2 \beta }{2\beta + d - d^G } }$ for $d^G = 0, \dots, \lfloor 2\beta \rfloor$. We require that $\mu$ is such that there exists an $H > 0$ such that for all $x$ there are positive constants $c_x$ with $c_x h^{d} < \mu( B_\mathcal{X} (x,h) ) $ for all $0 < h < H$. This occurs when $\mu$ admits a density with respect to the Lebesgue measure that is bounded away from 0 on every compact subset of $\mathcal{X} \subseteq \RR^d$, or more generally with respect to the natural Riemannian volume form on $\mathcal{X}$. \\

This paper will focus on the the  $(L, \beta)$-\textbf{H\"{o}lder class} for $\mathcal{F}$, which we now define. This class is useful because it provides a direct route to Proposition \ref{prop:close_groups} later. Recall that a function $f$ is \textbf{$\beta$-H\"{o}lder smooth} if $f$ has bounded partial derivatives up to order $k  = \lfloor \beta \rfloor$ (i.e., the largest integer not larger than $\beta$) that are bounded, and for which the $k^{th}$ partial derivatives are $(\beta - k)$-H\"{o}lder continuous (which means that there exists an $A$ such that $|a(x) - a(y)| \leq A d(x,y)^{\beta - k}$ for all $x,y \in \mathcal{X}$). We can then define the H\"{o}lder class as:
 \begin{equation}
     \mathcal{F}(L, \beta) = \Big\{   f :\in L^2( \mathcal{X} ) : \sum_{\alpha : 0 \leq  |\alpha| \leq k } \| \partial^\alpha f \|_\infty + \sum_{ \alpha : | \alpha | = k} \sup_{x \neq y \in \mathcal{X}} \tfrac{ | (\partial^\alpha f) (x) - (\partial^\alpha f) (y) |}{ d(x,y)^{\beta - k}} \leq L \Big\}
 \end{equation}
where the multi-indices $\alpha = ( \alpha_1, \dots, \alpha_d ) \in \ZZ_{\geq 0}^d$ have $|\alpha| = \sum_i \alpha_i$, and the partial derivatives $\partial^\alpha$  are given by $\partial^{\alpha_1} \partial^{\alpha_2}\cdots \partial^{\alpha_d}$. \\
\\
For any particular estimator $\hat{f}$, let $\mathcal{E}(\hat{f}, \mathcal{D}) = \EE( (\hat{f}(X) - Y )^2 \mid \mathcal{D} )$ where $X \sim \mu$ and $Y = f(X) + \epsilon$ be the $L^2$-\textbf{generalisation error}. We use the notation $\hat{\mathcal{E}}(\hat{f}) = \tfrac{1}{n} \sum_{i = 1}^n ( \hat{f}(X_i) - Y_i )^2 $ for the \textbf{empirical error}, an unbiased estimate of $\mathcal{E}( \hat{f}, \mathcal{X})$ when the data points summed over are from an independent copy of $\mathcal{D}$. We will typically be interested in the \textbf{risk} $R(\hat{f} ) = \EE( (\hat{f}(X) - f(X) )^2 )$ which is given by $\EE ( \hat{ \mathcal{E} }( \hat{f} ) ) - \sigma^2 $. 

\subsection{Differential Geometry}
\label{ssec:diff_geo}

Our proofs later will rely on the geometry of sub-manifolds of $\mathcal{X}$, so here recall the definitions used in differential geometry. Further details in a statistical context can be found in \cite{fletcher2010terse}. A \textbf{$d$-dimensional smooth manifold} is a Hausdorff and second countable topological space that is covered by a collection of open sets $U_a$ with homeomorphisms $\phi_a : U_a \rightarrow \RR^d$ such that the transition maps $\phi^b \circ \phi_a^{-1}$ are infinitely differentiable for all $a, b$. Examples are $\RR^d$ itself, the subset $[0,1]^d$, or the sphere $S^d = \{ x \in \RR^{d +1 } : \|x\|_2 = 1 \}$, well known in directional statistics \citep{mardia2009directional}. At each point $x$ in a smooth manifold $\mathcal{M}$ we can attach a copy of $\RR^d$ called the \textbf{tangent space}, written $T_x \mathcal{M}$, where each vector corresponds to an equivalence class of smooth curves $\gamma: [-1,1] \rightarrow \mathcal{M}$ with $\gamma(0) = x$ and with equal derivatives at $0$. Each curve $\gamma$ can be given a length by $L( \gamma ) = \int_{-1}^1 \| \gamma'(t) \|_2 \dd t$, where the norms are taken in the tangent space at $\gamma(t)$. This in turn given a metric space structure on $\mathcal{M}$, where $d_\mathcal{M}( x, y) = \inf_\gamma L (\gamma) $ where the infimum is taken over all piece-wise smooth curves $\gamma$ with $\gamma(-1) = x \in \mathcal{M}$ and $\gamma(1) = y \in \mathcal{M}$. \\
\\
We call a subset $\mathcal{M}' \subseteq \mathcal{M}$ an \textbf{immersed submanifold} if the inclusion map $\phi : \mathcal{M}' \rightarrow \mathcal{M} : x \mapsto x$ has an everywhere injective derivative. An immersed submanifold is called \textbf{embedded} if $\phi$ is also a homeomorphism (i.e., it preserves topology). Every smooth manifold $\mathcal{M}$ can be isometrically embedded in $\RR^{k}$ for some $k > d$, a result known as the Nash Embedding Theorem \cite{nash1956imbedding}. This means that curves through $\mathcal{M}$ have the same length as they would when integrated in the embedding space, and so the distance between points in the manifold is bounded from below by the distance in the embedding Euclidean space.

\subsection{Group Theory}
\label{ssec:group_theory}

A \textbf{group} is a set $G$ with an associative binary operation (or \textbf{multiplication}) $(a, b) \mapsto ab$, such that: there is an identity element $e \in G$ with $eg = ge = g$ for all $g \in G$; and there are inverses $g^{-1}$ with $g^{-1} g = g g^{-1} = e$ for all $g \in G$. We call a group with one element, $\{e\}$, the trivial group and denote it by $I$. A group acts on the domain $\mathcal{X}$ via a map $\cdot : G \times \mathcal{X} \rightarrow \mathcal{X}$ with the properties that $e \cdot x = x$ and $g \cdot( h \cdot x ) = (gh) \cdot x$ for all $g, h \in G$ and $x \in \mathcal{X}$. This means that the maps $g \cdot : \mathcal{X} \rightarrow \mathcal{X}$ are the bijections $\phi_g$ from the introduction. We will usually assume that all group actions are \textbf{faithful}, which means that if $g \cdot x = x$ for all $x \in \mathcal{X}$ then $g = e$. A group action partitions the domain $\mathcal{X}$ into \textbf{orbits} of the form $[x]_G = \{ g \cdot x : g \in G \}$. We call the set of orbits the \textbf{quotient space} of the action, and write it as $\mathcal{X} / G$. \\
\\
If the group $G$ has a topology that agrees with the group multiplication, in the sense that $(g, h) \mapsto g^{-1}h$ is continuous, then we say that $G$ is a \textbf{topological group}. When we use topological adjectives (e.g. compact, Hausdorff, etc.) for a group we assume that the group and topology agree. We consider only \textbf{locally compact} groups in this paper, which means that every element $g \in G$ has a compact neighbourhood $g \in U \subseteq N$ for some open $U$. A stronger condition is that $G$ (as a set) also admits a differential manifold structure, which must also agree in that the map above is smooth. We call such a group a \textbf{Lie group}. When we use the notion of dimension of a group we are referring to a Lie group $G$. As a manifold, it must necessarily be Hausdorff and second countable, and admit a metric (in the sense of distance) $d_G : G \times G \rightarrow [0, \infty)$.  \\
\\
If $G$ is a locally compact group then there exists a Borel measure $\Gamma$ on $G$ such that $\Gamma( A ) = \Gamma( \{ ga : a \in A \} )$ for all measurable $A$ and all $g \in G$. Moreover, this measure is unique up to scaling, and we call such $\Gamma$ a \textbf{(left) Haar measure} on $G$ \citep{haar1933massbegriff}. When $G$ is compact, we can normalise any left Haar measure to give a uniform distribution on the group, $U(G)$.  An example of this for the group of unit norm complex numbers under multiplication is precisely the uniform distribution on this set. The locally compact group $\RR^d$ under vector addition has the Lebesgue measure as a Haar measure.  \\
\\
If a topological group acts on $\mathcal{X}$ we will usually assume that the action is continuous (in the product topology on $G \times \mathcal{X}$), and when a Lie group acts we will usually assume that the action is smooth. In the case of a locally compact Lie group acting smoothly and faithfully on $\mathcal{X}$ (the primary focus of this paper), the orbits $[x]_G$ are themselves immersed sub-manifolds of $\mathcal{X}$. If the action is also \textbf{proper}, i.e., inverse images of compact sets under the map $(g, m) \mapsto (m, g\cdot m)$ are compact, then the orbits are in fact smoothly embedded submanifolds. The Principle Orbit Type Theorem \citep{dieck1987transformation} says that there is a ``generic'' topology for orbits of a smooth proper action of a compact group $G$, i.e., that there is an open dense subset $U \subseteq \mathcal{X}$ with orbits $[x]_G$ that are pairwise diffeomorphic for all $x \in U$. In particular, the dimension of these principle orbits is the reduction we can achieve by symmetrising the estimator, which we denote $d^G$ with a superscript to distinguish from a metric on the group with a subscript $d_G$.

\subsection{Invariant Functions and Function Spaces}
\label{ssec:inv_funcs}
Let $G$ be any group acting on $\mathcal{X}$. We say that a function $f \in L^2(\mathcal{X})$ is \textbf{$G$-invariant} if there is a function $f_0$ with $\| f_0 - f\|_2 = 0$ and:
\begin{equation}
    f_0(x) = f_0( g \cdot x)
\end{equation}
for all $g \in G$ and $x \in \mathcal{X}$. If $\mathcal{F}$ is a set of functions in $L^2(\mathcal{X})$, we let $\mathcal{F}_G$ be the subset of $G$-invariant functions in $\mathcal{F}$. For any group $G$ the space $L^2_G(\mathcal{X})$ is a linear subspace of $L^2(\mathcal{X})$. This subspace is also closed in the topology of $L^2(\mathcal{X})$. If $G$ is acting continuously on $\mathcal{X}$, then we also have the following equivalence between $G$- and $\overline{G}$-invariant functions (with the bar indicating the topological closure of $G$).
\begin{Lem}[\cite{christie2023estimating}]
    \label{lem:closure_invariance}
    Suppose that $\mathcal{G}$ is a locally compact group acting continuously on $\mathcal{X}$. Any continuous $f \in L^2(\mathcal{X})$ is $G$-invariant if and only if $f$ is $\overline{G}$-invariant, for all $G \leq \mathcal{G}$. 
\end{Lem}

Moreover, for every group $\mathcal{G}$ acting on $\mathcal{X}$, every function $f \in L^2(\mathcal{X})$ has a unique \textbf{maximal invariant subgroup} \citep{christie2023estimating}, i.e., a unique group $G_0 \leq \mathcal{G}$ such that $f$ is $G_0$-invariant, and every subgroup $f$ is invariant to is a subgroup of $G_0$. To see this, simply consider the group $\overline{ \langle G \leq \mathcal{G}: f \text{ is } G\text{-invariant} \rangle }$. We write $G_{\max}(f, \mathcal{G} )$ for this maximal invariant subgroup. 

\subsubsection{Compact Symmetries}
When $G$ is compact and $\mu$ is a $G$-invariant probability measure, $L^2(\mathcal{X})$ admits an orthogonal decomposition:
\begin{equation}
    \label{eq:orth_decomp_1}
    L^2(\mathcal{X}) = L^2_G(\mathcal{ X}) \oplus A_G( \mathcal{X} )
\end{equation}
See \cite{elesedy2021provablyLin} for a proof. We term the functions in $A_G( \mathcal{X} )$ \textbf{$G$-anti-symmetric}. There also exists an orthogonal projection onto $L^2(\mathcal{X})$, denoted $S_G : L^2(\mathcal{X}) \rightarrow L^2_G( \mathcal{X} )$. This projection has an explicit form:
\begin{equation}
    \label{eq:compact_projection}
    (S_G f) (x) = \EE ( f(g \cdot x ) )
\end{equation}
where $g \sim U( G )$. The orthogonal decomposition in equation \ref{eq:orth_decomp_1} can therefore be expressed as:
\begin{equation}
    \label{eq:orth_decomp}
    f = S_G f + f_G^\perp, \quad \text{ where } f_G^\perp = f - S_G f \in A_G (\mathcal{X} ).
\end{equation}
Note that by definition $S_G f = 0$ for any $f \in A_G( \mathcal{X} )$. When $\mu$ is not $G$-invariant, this operator is no longer an orthogonal projection, but it is still idempotent and its image is still $L^2_G( \mathcal{X} )$, so is still an oblique projection onto the $G$-invariant sub-space with eigenvalues bounded by $1$.

\begin{eg}
    \label{eg:example_decomp}
    Consider the case of $G = \{ z \in \CC : |z| = 1 \}$ acting by rotations on $\RR^2$ with $\mu$ a standard Gaussian measure.  The function $f$ on this disc given by: 
    \begin{equation}
        f(x) = \sin ( \| x \|_2 ) \cos( 2 \mathrm{arctan}( x_2 / x_1 ) )^2 = \sin (r ) \cos( 2 \theta )^2 
    \end{equation}
    in Euclidean and polar coordinates respectively, has an orthogonal decomposition shown in figure \ref{fig:example_decomp}.
    
    \begin{figure}[h!]
        \centering
        \begin{tabular}{ m{3cm} m{0.5cm} m{3cm} m{0.5cm} m{3cm} }
        \centering \includegraphics[scale = 0.24]{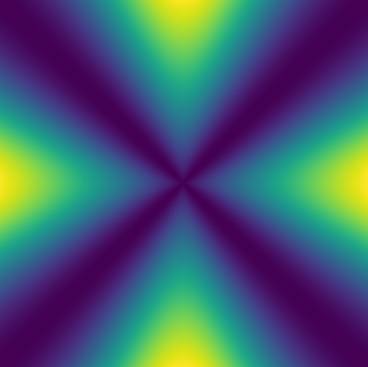} & 
        \centering = & 
        \centering \includegraphics[scale = 0.24]{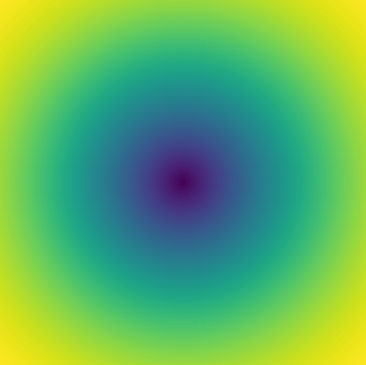} & 
        \centering + & 
        \centering \includegraphics[scale = 0.24]{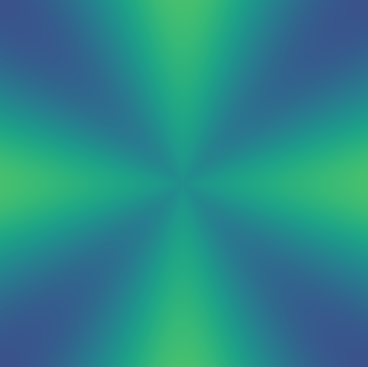} \tabularnewline
        \centering $f$ & 
        \centering = & 
        \centering $S_{G} f$ & 
        \centering + & 
        \centering $f^\perp_{G} $ \tabularnewline
        \tabularnewline
        &
        \centering = & 
        \centering $\frac{1}{2} \sin( r  )$ & 
        \centering + & 
        \centering $\sin(  r  ) (  \cos( 2 \theta )^2  - \frac{1}{2} )$
        \end{tabular}
        \caption{The orthogonal decomposition of $f = \sin(r) \cos( 2 \theta)^2$ into a $G$ symmetric piece and a $G$-anti-symmetric piece, as described in example \ref{eg:example_decomp}. }
        \label{fig:example_decomp}
    \end{figure}
\end{eg}

\subsubsection{Non-Compact Symmetries}
When $G$ is not compact, we can no longer use the operator in equation \ref{eq:compact_projection} because a Haar measure may not be normalisable (as with the Lebesgue measure on $\RR^d$). We can however pick any other distribution $\rho$ on $G$, and generalise to the \textbf{Partial Symmetrisation Operator} $S_{\rho} : L^2(\mathcal{X}) \rightarrow L^2(\mathcal{X})$ given by
\begin{equation}
    \label{eq:sym_contraction}
    S_\rho f ( x) = \EE( f( g \cdot x) )
\end{equation}
where $g \sim \rho$. This operator is not a projection, because it is not idempotent, but it is still a contraction in the $L^2$ norm when the integrating measure $\mu$ is $G$-invariant, and in the $L^\infty$ norm for all measures $\mu$. It is also trivially the identity on $L^2_G( \mathcal{X} )$.

\section{Estimation using a Particular Symmetry}
\label{sec:sym_rates}

We now turn to the statistical problem of estimating $f$ from within the H\"{o}lder function class $\mathcal{F}(L, \beta) \subseteq L^2(\mathcal{X} )$ using an iid sample $\{ (X_i, Y_i ) \}_{i = 1}^n$ with $\EE( Y_i \mid X_i) = f(X_i)$ and $\mathrm{Var}( Y_i \mid X_i ) = \sigma^2 < \infty$. Let $G$ be any locally compact Lie group acting smoothly, properly, and faithfully on $\mathcal{X}$ and let $U \subseteq G$ be a compact set with nonempty interior containing the identity. Suppose that $f_n$ is an estimator for $f$ that is minimax optimal for the class $\mathcal{F}(L, \beta)$, that we will call the \textbf{base estimator}. We will first, for each point $x\in \mathcal{X}$, construct a discrete distribution $\rho_m^x$ with support contained in $U$ and consider the partial symmetric estimator $S_{\rho_m^x} f_n(x)$. We will establish point-wise bounds: 
\begin{equation}
    \sup_{f \in \mathcal{F}(L, \beta)} \EE\big(  ( S_{\rho_m^x} f_n(x) - S_{\rho_m^x} f(x) )^2 \big) \leq C_{G, x} \phi_n^{\beta, d - d^G }
\end{equation}
where $\phi_n^{\beta, d - d^G } = n^{ - \frac{2 \beta}{2\beta + d - d^G }}$ and $d^G = \dim [x]_G$ for any principle orbit $[x]_G$. We will then show that these can be integrated under some conditions on the group action. Note that this estimator will not converge to $f$ unless $f$ is not $G$-invariant so these are not true risk bounds, but they are useful for understanding the convergence of these estimators to their asymptotic bias. \\
\\
For $G$-invariant $f$, this is perhaps the expected rate that would be achieved by projecting the covariates to the quotient space $\mathcal{X} / G$ and regressing the response against these orbits (of course, proving this would require a mild treatment of the topological conditions of the quotient space $\mathcal{X} / G$ and then applications of standard theorems). For the non $G$-invariant $f$, it is non-obvious that we can apply usual results because we do not observe iid data with conditional mean $S_{\rho_{m}^x} f$ - the distribution is very much affected by the averaging over the orbits.

\subsection{Defining Partially Symmetrised Estimators}
\label{ssec:par_sym_ests}

In this section we give an explicit construction for Partially\footnote{The name \textit{partial} comes from the fact that $S_{\rho_m^x}$ may not project $f_n$ to the space of $G$-invariant functions, as noted in section \ref{ssec:inv_funcs}.} Symmetrised Estimators, as described above. The main choice to make is the distribution $\rho$ to be used to partially symmetrise the base estimator $f_n$. We will do this by constructing a set of $m$ points $g^i \in G$ and take $\rho = U( \{ g^i \} )$. This must balance two goals: we want as many points as possible around $[x]_G$ in order to maximise the variance reduction, but would like the estimators $f_{n} ( g^i \cdot x )$ to be independent for theoretical tractability. In order to achieve this, we pick the a large number of points $m$ such that we can guarantee that $B_\mathcal{X}( g^i \cdot x, h)$ are pairwise disjoint for all $i, j$ for a chosen bandwidth $h > 0$. This is done by examining the geometry of the orbit $[x]_G$, which is a smoothly embedded sub-manifold of $\mathcal{X}$ under our conditions on the group action. These $g^i$ depend on $x$ and $h$, so we add subscripts to remind us of this choice.\\ 
\\
To construct the set $\{ g^i_{x,h} \}$, first isometrically embed $\mathcal{X}$ in $\RR^q$ for some $q \in \NN$ and extend $d_\mathcal{X}$ to the euclidean metric on $\RR^q$, and consider all points in $\mathcal{X}$ and $T_x[x]_G$ via this embedding. Since $U \subseteq G$ is a compact neighbourhood of the identity, $[x]_U = \{ g \cdot x \in \mathcal{X} : g \in U \}$ is a compact sub-manifold of $[x]_G$ of the same dimension. Let $W_{x}^{[x]_U}$ be the largest hypercube in $T_x [x]_U$ that is contained in the orthogonal projection of $[x]_U$ onto $T_x[x]_U$. Let $R_x^{[x]_U}$ be the side length of the hypercube $W_x^{[x]_U}$. If the orbit is $0$-dimensional, as with the trivial group $I = \{ e \}$, then we take $R_x^{[x]_U} = 1$. We construct the group elements $g_{x,h}^i$ as follows:
\begin{enumerate}[  1)]
    \item Pick a maximal grid $\{ a_i \}_{i = 1}^m$ in $W_x^{[x]_U}$ such that $d(a_i, a_j) \geq 2h $ whenever $i \neq j$;
    \item Orthogonally (with respect to $T_x[x]_U$) project this grid onto $[x]_U$, giving a set $\{ u_i \}_{i = 1}^m \subseteq [x]_U$; and
    \item Pick any\footnote{note that there can be more than one option if the \textbf{stabiliser} $G_x = \{ g \in G :g \cdot x = x \}$ is non trivial.} $g_{x,h}^i \in U$ such that $g_{x,h}^i \cdot x = u_i$.
\end{enumerate}

\begin{figure}[h]
    \centering
    \begin{tikzpicture}
        \draw (0,0) circle (2);
        \draw[<->] (2,-2.5) -- (2,2.5) node[ left]{$T_x [x]_U$};

        \draw[fill] (2,2) circle (0.07cm) node[right]{$a_1$};
        \draw[fill] (2,1) circle (0.07cm) node[right]{$a_2$};
        \draw[fill] (2,0) circle (0.07cm) node[right]{$a_3$};
        \draw[fill] (2,-1) circle (0.07cm) node[right]{$a_4$};
        \draw[fill] (2,-2) circle (0.07cm) node[right]{$a_5$};

        \draw[|-|, dashed] (2.6, 0) -- (2.6,1);
        \node[right] at (2.6, 0.5 ) {$2h$};

        \draw[|-|, dashed] (3.2,-2) -- (3.2,2);
        \node[right] at (3.2,0) {$R_x^{[x]_U}$}; 

        \draw[fill] (0,2) circle (0.07cm) node[below]{$g^1_{x,h} \cdot x$};
        \draw[fill] (1.732,1) circle (0.07cm) node[left]{$g^2_{x,h} \cdot x$};
        \draw[fill] (2,0) circle (0.07cm) node[left]{$g^3_{x,h} \cdot x$};
        \draw[fill] (1.732,-1) circle (0.07cm) node[left]{$g^4_{x,h} \cdot x$};
        \draw[fill] (0,-2) circle (0.07cm) node[above]{$g^5_{x,h} \cdot x$};

        \node at (-2.5,0) {$[x]_G$};
        \draw[dotted] (0,2) -- (2,2);
        \draw[dotted] (0,-2) -- (2,-2);
        \draw[dotted] (1.732,1) -- (2,1);
        \draw[dotted] (1.732,-1) -- (2,-1);
    \end{tikzpicture} 
    \caption{ The orbit of $x \in \RR^2$ under the action of $G = SO(2)$, along with the points $\{ a_i \}$ and $\{ u_i = g_{x,h}^i \cdot x \}$ that will be used to partially symmetrise the estimator $f_n$.}
    \label{fig:e}
\end{figure}
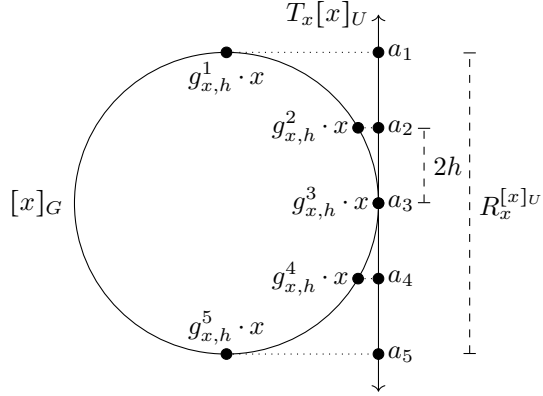

This allows us to define the symmetrising distribution $\rho_m^x$ as the uniform distribution on the set $\{ g_{x,h}^i \}$, which has size $m$ (which depends on $x$, $h$, and $U$ implicitly). This process is depicted in figure \ref{fig:e} for an example with $G$ the special orthogonal group of 2D rotations, $SO(2)$. We use this choice of $\rho$ to construct the partial symmetrisation operator $S_{\rho_m^x}$, and apply this to the chosen base estimator $f_n$ to construct the \textbf{Partially Symmetrised Estimator} $S_{ \rho_m} f_n$. The first step to analysing the partially symmetrised estimator requires us to prove the following lower bound on $m$ and confirm that the points are sufficiently spaced around $[x]_U$.  

\begin{prop}
    \label{prop:packing_nos}
    For all $x \in \mathcal{X}$ and $h > 0$ we have $m \geq \max( 1 , ( R_x^{[x]_U} (2h)^{-1} )^{\dim [x]_U } )$, and $d_\mathcal{X}( g^i_{x,h} \cdot x, g^j_{x,h} ) \geq 2 h$ whenever $i \neq j$.
\end{prop}

\subsection{Point-wise Convergence Rates of Partially Symmetrised Estimators}
\label{ssec:point_wise_risk}
Symmetrised estimators improve convergence because symmetrisation operators can have minimal effect on the bias (even none when $f$ is $G$-invariant), but the variance can be dramatically reduced. This allows for a re-optimisation of the bias-variance trade-off for a faster converging estimator. This is formalised in this section, which requires us to make some assumptions on the base estimator $f_n$, the bandwidth $h_n$, and later on the group action. \\
\\
\textbf{Assumption set (E):}
\begin{enumerate}[  (E1)]
    \item (Strict Locality) Conditioned on $\mathcal{D}_X$,  $f_n(x)$ and $f_n(y)$ are independent whenever $d(x,y) \geq 2h$;
    \item (Default Value) If the ball $B_\mathcal{X}(x,h)$ contains no $X_i$, then $f_n(x) = 0$;
    \item (Optimality) There exist positive constants $B, V, C \in \RR_{>0}$, such that for all $f \in \mathcal{F}$, for $\mu$-almost all $x$, and for all $n$:
        \begin{enumerate}
            \item (Bias) If $B_\mathcal{X}(x,h) \cap \mathcal{D}_X$ is non-empty, then $| \EE( f_n(x) \mid \mathcal{D}_X ) - f(x) | \leq B h^\beta $;
            \item (Variance) If $B_\mathcal{X}(x,h) \cap \mathcal{D}_X$ is non-empty, then $\mathrm{Var}( f_n(x) \mid \mathcal{D}_X ) \leq \tfrac{ V }{ | B_{\mathcal{X}}(x,h) \cap \mathcal{D}_X | } $;
            \item (Risk) $\EE( (f_n(x) - f(x) )^2 ) \leq C \big( B^2 h^{2\beta} + V (n \mu( B_\mathcal{X}(x,h) ) )^{-1} \big)$
        \end{enumerate}
\end{enumerate}
These assumptions on the base estimator are natural. (E1) is somewhat restrictive but is satisfied by local polynomial estimators using a kernel $K$ supported only on the interior of the unit ball that is scaled by a bandwidth $h$, and with a deterministic value for the case where there is no data in the $h$-ball around $x$. When the deterministic default value is 0, the local polynomial estimator satisfies (E2) trivially. Any other estimator can be made to satisfy this too by multiplication by the indicator function $\mathbf{1}_{|\mathcal{D}_X \cap B_\mathcal{X}(x,h)| > 0}$. Condition (E3) is required to ensure that the derived estimators are also rate optimal, and essentially says that the estimator $f_n$ performs well when there is data in the $h$-ball and that the sampling regime ensure that this happens frequently. These assumptions are satisfied even for simple estimators such as the Nadaraya-Watson estimator when $\beta \leq 2$ in the context of our statistical problem (section \ref{ssec:stat_problem}); the full details of this can be found in appendix \ref{app:lce_properties}. 

\begin{rem}
    Assumption (E1) could be weakened to allow for sufficiently quickly decaying covariance which would allow for Gaussian kernels or $k$-nearest neighbour estimators. Such a weakening is offered with condition (E1b'), which is satisfied for a Gaussian kernel in local polynomial estimation, or when the bandwidths are chosen adaptively. An interesting question would be whether such a condition also applies to kernel ridge regressors or neural networks. 
\begin{enumerate}[  (E1b')]
    \item (Weak Locality) $\mathrm{Cov}( f_n(x), f_n(y) \mid \mathcal{D}_X ) \leq c \exp( - a n )$ for some positive constants $a, c \in \RR_{> 0}$.
\end{enumerate}
\end{rem}

We will also require the following conditions that relate the geometry of the (partial)-orbits $[x]_U$ to the measure $\mu$ and to the smoothness of the regression function.   \\
\\
\begin{minipage}{\textwidth}
\textbf{Assumption set (Q):} 
\begin{enumerate}[  (Q1)]
    \item $\EE\big( (R_X^{[X]_U})^{-d^G } \big) < \infty$, where $X \sim \mu$; and
    \item $d^G \leq 2 \beta$.
\end{enumerate}
\end{minipage}
\vspace{0.1cm}

Condition (Q1) says that the distribution $\mu$ is not ``too concentrated'' around any singularities of the action. This is satisfied in many normal situations, as shown in the example below. We also show that in some cases this is not satisfied. The second condition, (Q2), ensures that the new optimal bandwidth selection rule does not cause the variance of $f_n$ to increase too rapidly. When $d^G > 2 \beta$ the bound on the variance of the original estimator $f_n$ in (E2) can in fact increase with $n$. This is balanced by averaging over $O( h^{- d^G } )$ more points, resulting in the overall reduction in variance. However, if $d^G > 2 \beta$ then we are saying that the number of points in each local ball $B(x,h)$ will decrease with $n$ rather than increase.   

\begin{eg}[Example of Assumption (Q1)]
    \label{eg:R_x}
    Consider the action of $G = SO(3)$ on $\mathcal{X} = \RR^3$ by rotations, and since $G$ is compact we take $U = G$. The orbit of $x \in \mathcal{X}$ is the sphere of radius $\| x \|_2$, which means that $R_x^{[x]_G} = \sqrt{2} \| x \|_2$ for all $x \in \RR^3$. We will consider three possible distributions of $X$ all centered at the singularity of this action, the origin. If $\mu$ is a standard Gaussian distribution on $\mathcal{X}$, i.e., $X \sim N( 0 , I_3 )$, then $\| X \|_2^2 \sim \chi_3^2$. The regular points of the action are the non-zero points of $\mathcal{X}$, and $d^G = 2$. Thus 
    \begin{equation}
        \EE( (R_X^{[X]_G})^{- d^G} ) = \EE( ( \sqrt{2} \| X \|_2 )^{-2} ) = 1 / 2 < \infty
    \end{equation}
    using the fact that if $Y \sim \chi^2_\nu$ then $\EE( 1 / Y ) = 1/ (\nu - 2)$. If $\mu$ is the uniform distribution on the unit ball then $\| X \|_2 \overset{D}{=} z^{1/3}$ when $z \sim U( [0,1] )$ and $\| X \|_2^{-2} = 3 < \infty$ too.  If instead the law of $X$ has $\| X_2 \|_2^2 \sim \mathrm{Cauchy}_+( 0, 1 )$, where $Y \sim \mathrm{Cauchy}_+(0,1)$ if it is the absolute value of a $\mathrm{Cauchy}(0,1)$ variable, and is uniform in direction, then the expectation above is infinite. 
\end{eg}

These assumptions allow us to prove the following proposition bounding the point-wise convergence rate of the partially symmetrised estimator $S_{\rho_x^m} f_n (x)$. This is done by separate analysis of the point-wise bias and variance of this estimator, and the results for each of these can be found in Lemmas \ref{lem:pointwise_bias_disc_sym} and \ref{lem:pointwise_var_disc_sym} in section \ref{sec:proofs} along with the proof of this proposition.

\begin{prop}[Point-wise MSPE of $S_{\rho_m^x} f_n$] 
    \label{prop:point_wise_error_bound}
    Under assumption sets (E) and (Q), in the random design context, with the bandwidth choice $h_n = a n^{-1/ ( 2\beta + \dim{\mathcal{X}} - d^G ) }$ for any positive $a$, we have:
    \begin{align}
		\sup_{ f \in \mathcal{F}(L, \beta) } \EE \big( ( S_{\rho_m^x} f_n (x) - S_{\rho_m^x} f(x))^2 \big) &\leq C_{G,x} \phi_n^{\beta, d - d^G }
	\end{align}
 for all $n \in \ZZ_{> 0}$, where $C_{G,x}$ is given by:
 \begin{equation}
     C_{G,x} =  \Big( 2 \big( B^2  a^{2 \beta}  + \tfrac{ 2^{d^G}V }{ c_{*}(x)  ( R_{ x}^{[x]_U})^{d^G } } \big) + ( L + V ) \big( \tfrac{c_{*}(x) a^{d} ( 2 \beta - d^G ) }{ 16 \beta } \big)^{ \tfrac{ 2 \beta }{d}} \Big)
 \end{equation}
 where $c_*(x) = \inf_{g \in U} c_{ g \cdot x}$. In the fixed design regime these bounds hold for all sufficiently large $n$ such that $|B_\mathcal{X}(x, h) \cap \mathcal{D}_X | > n \mu( B_\mathcal{X}(x,h) ) / 2$.
\end{prop}

\subsection{Integrated Risk of Partially Symmetrised Estimators}
\label{ssec:integrated_risk}

 The pointwise bound is true for almost all $x \in \mathcal{X}$, but not uniform because of the $R_x^{[x]_U}$ and $c_*$ terms. The inverse moments of $R_X^{[X]_U }$ can blow up to infinity on neighbourhoods of singular points of the action. Assumption (Q1) controls this blow up, ensuring that $\mu$ is not too concentrated around the singularities of the action. This allows us to integrate the pointwise bound to gain integrated rates over some Borel subset $\Omega \subseteq \mathcal{X}$. When $\Omega \subseteq \mathcal{X}$ is compact, we can take $c_*$ to be a constant in $x$ over $\Omega$, because its minimum over $\mathcal{X}$ works in the required bound. Thus we can obtain the integrated risk in the following proposition. Note that from here we sometimes drop the superscript $x$ on $\rho_m^x$ when considering $S_{\rho_m^x} \phi (x)$ as a function of $x \in \mathcal{X}$.

\begin{thm}[IMSPE of $S_{\rho_m} f_n$] 
    \label{thm:MISPE_tilde_f}
    Suppose that $\Omega$ is a compact Borel subset of $\mathcal{X}$ such that (Q2) holds when $X$ is conditioned on $X \in \Omega$, and suppose that $c_x = c > 0$ for all $x \in \Omega$. Under assumption sets (E) and (Q), with a bandwidth $h_n = a n^{-1/ (2\beta + d - d^G) }$ we have:
    \begin{equation}
		\sup_{f \in \mathcal{F}(L, \beta) } \EE\big( (S_{\rho_m^X} f_n(X) - S_{\rho_m^X} f(X))^2 \mid X \in \Omega \big)  \leq C_{G, \Omega} \, \phi_n^{\beta, d - d^G }
	\end{equation}
    where $x$ is any regular point of the action, and $C_{G, \Omega} = \EE( C_{G,X} \mid X \in \Omega ) < \infty$ where $X \sim \mu$. In particular, when $\mathcal{X}$ is compact, we have:    
    \begin{equation}
		\sup_{f \in \mathcal{F}(L, \beta) } \EE\big( \| S_{\rho_m} f_n - S_{\rho_m} f \|_2^2 \big)  \leq C_G \, \phi_n^{\beta, d - d^G }
	\end{equation}
 where $\EE( C_{G,X} ) < \infty$.
\end{thm}

The condition that assumption (Q2) holds when conditioned on $X \in \Omega$ holds in almost all interested cases, for example, a $d$-dimensional unit hypercube in $\RR^d$ when $\mathcal{G} = \RR^d$ acting by translation.

\subsection{Computing Partial Symmetrisation Operators}

The explicit construction of $\rho_m^x$ is always possible given any point $x \in \mathcal{X}$, the bandwidth $h$, the group $G$, and its action on $\mathcal{X}$, however, it requires computational effort in the general case---especially when the orbits are non-linear. We now show that in some cases we can compute these averaging operators quickly by using the specific structure of the group action.

\subsubsection{Covariate Sparsity - Linear Orbits}

Recall example \ref{eg:cov_sparsity}, where we show that covariate sparsity in non-parametric regression is a special case of a translation symmetry. Specifically, in this case $\mathcal{G} = \RR^d$ acts on $\mathcal{X} = \RR^d$ by translation and the covariates that $f$ does not depend on can be identified with subgroups of $\mathcal{G}$ of the form $\{ a e_i : a \in \RR \}$ for the standard basis $\{ e_i \}$ of $\mathcal{X}$. When constructing symmetrisation operators we can instead exploit the uniform structure of the orbits of all subgroups of $\mathcal{G}$. Let $U = [-1, 1]^d$, and see that $T_x[x]_U$ can be identified with $U$ precisely (as there are only trivial stabilisers) for all $x \in \mathcal{X}$. Thus a $(2h)$-grid on $[x]_U$ is exactly the easily computable grid for the linear tangent space. This means that we need only one computation of the group elements $g^i$ for the dataset $\mathcal{D}$. 

\subsubsection{Compact Groups}

\label{sssec:compact_groups_comps}
In the special case of a compact group $G$ and when $\mu$ is $G$-invariant, we can use the properties of the full symmetrisation operator $S_G$ to improve on the estimator $S_{\rho_m} f_n $. Specifically, we can show $S_G f_n (x) = S_G S_{\rho_m^x} f_n (x)$ by the shift invariance of the Haar measure (see the proof below), which gives the result:
\begin{equation}
    \| S_G f_n (x) - S_G f \|_2 = \| S_G S_{\rho_m^x} f_n - S_G S_{\rho_m^x} f \|_2 \leq \| S_{\rho_m^x} f_n - S_{\rho_m^x} f \|_2
\end{equation}
because $S_G$ is an orthogonal projection with operator norm 1. Thus the fully symmetrised estimator $S_G f_n$ will outperform the partially symmetrised $S_{\rho_m} f_n$, almost surely. Thus as a corollary to theorem \ref{thm:MISPE_tilde_f}, we obtain:

\begin{Cor}[IMSPE of Fully Symmetrised Estimators]
    \label{cor:MSIPE_of_syms}
    Suppose that $\mathcal{X}$ and $G$ are both compact and that $\mu$ is $G$-invariant. Under conditions (E) and (Q), we have that:
	\begin{equation}
		\sup_{f \in \mathcal{F}(L, \beta)} \EE( \| S_G f_n  - S_G f \|_2^2 ) \leq C_G \phi^{ \beta, d - d^G }_n
	\end{equation}
    for the same constant $C_G$ as in Theorem \ref{thm:MISPE_tilde_f}.
\end{Cor}

Now we can estimate $S_G f_n$ with a simple Monte-Carlo estimator, where we average over $M$ samples of $g_i \iid U(G)$. This gives the estimator of $f$:
\begin{equation}
    \hat{f}_{n, G} (x) = \frac{1}{M} \sum_{i = 1}^M f_n ( g_i \cdot x) 
\end{equation}
which is easily computable whenever we can sample uniformly from the group. The analysis of this estimator is simple when $\mu$ is $G$-invariant, and achieves the same rates as $S_{\rho_m} f_n$ and $S_G f_n$ when $M = n$.

\begin{Cor}
    \label{cor:MSIPE_f_hat_G}
    Suppose $\mathcal{X}$ and $G$ are compact, and that $\mu$ is $G$-invariant. Under assumption sets  (E) and (Q),  with $M \asymp n$, the estimator $\hat{f}_{n,G}$ satisfies:
    \begin{equation}
       \sup_{f \in \mathcal{F}(L, \beta) } \EE ( \| \hat{f}_{n,G } - S_G f \|_2^2 ) \leq (12 C + 12 L^2 + 2 C_G )  \phi_n^{\beta, d - d^G }
    \end{equation}
    where $C_G$ is as in Proposition \ref{thm:MISPE_tilde_f} and $C$ is as in (E3).
\end{Cor}

\section{Adaption to Maximal Symmetries}

We can now consider the main problem of this paper: what is the best subset $G \subseteq \mathcal{G}$ to use to (partially) symmetrise $f_n$ to best estimate $f$? The first step is to note that if $f$ is invariant to the elements of a subset $A \subseteq \mathcal{G}$, then $f$ is also the group \textbf{generated} by $A$, i.e., the intersection of all subgroups of $\mathcal{G}$ that contain $A$, written $\langle A \rangle$. Moreover Lemma \ref{lem:closure_invariance} ensures that $f$ will be invariant to the topological closure $\overline{ \langle A \rangle }$. Thus our best subset will be a closed subgroup of $\mathcal{G}$---a fact with simplifies our search remarkably to just the question: How can we estimate the maximal invariant subgroup of $f$, and can we do this quickly enough to improve the estimation of $f$?    \\
\\
Given any estimator $f_n$ that satisfies the condition set (E), the rates in the previous section hold for the distributions $\rho_x^m(G)$ constructed point-wise in $x$ for each closed subgroup $G \leq \mathcal{G}$. So a simple two step $M$-estimator, where we minimise $\hat{\mathcal{E}} ( S_{\rho_m^x(G)} f_{n} ) $ over the closed subgroups $G \leq \mathcal{G}$, would give an estimate of the unique maximal subgroup of $f$, and the minimising estimator might be expected to perform with the rate $\phi_n^{\beta, d - d^{G_{\max}(f, \mathcal{G})} }$. In this section we should that it can be approximated by minimising over a well chosen finite set of subgroups (that grows with $n$), and that the minimiser over this set accomplishes the desired property of a risk that decays with the fast rate $\phi_n^{\beta, d - d^{G_{\max}(f, \mathcal{G})} }$, both globally when $\mathcal{X}$ is compact and locally when it is not.  \\
\\
This first requires three key pieces of analysis in order to do statistics when the objects are subgroups: 
\begin{enumerate}[  1)]
    \item that we can construct a metric on the set of closed subgroups in the same way that we must with other object oriented data analysis tasks \citep{marron2021object};
    \item that this set is totally bounded in this metric\footnote{ recall that a subset $U$ of a metric space $(X, d_X)$ is \textbf{totally bounded} if it can be covered by a finite number of open $\epsilon$ balls $B_X( x_i, \epsilon)$ for any $\epsilon > 0$.}; and 
    \item that we can bound the asymptotic bias of $f_{n, G}$ in terms of the distance between $G$ and $H$ (for which $f$ is $H$-invariant).
\end{enumerate}
We denote the set of closed subgroups of $\mathcal{G}$ as $K(\mathcal{G})$. Moreover, we will be greedy in our estimation by further reducing the problem to just connected subgroups. This is justified because if $f$ is $G$-invariant, then $f$ is also $G'$-invariant where $G'$ is the connected component of the identity in $G$, and $d^{G'} = \dim [x']_{G'} = \dim [x]_G$ for the regular points of each actions.  \\
\\
With these mathematical results we can define our estimator of $G_{\max}( f, \mathcal{G} )$, the Error Minimising Symmetry, and then our estimator for $f$, the Best Symmetric Estimator given by partial symmetrisation by the error minimising symmetry. We then prove that the Best Symmetric Estimator achieves the desired adaption to the maximal symmetry of $f$.

\subsection{Metrising Subgroup Space}
\label{ssec:subgroup_space}

First, recall that if $(X, d_X)$ is a metric space then the \textbf{Hausdorff metric} between compact subsets $A, B$ is given by
\begin{equation}
	d_{\mathrm{Haus}} ( A, B ) = \max \big( \sup_{g \in A} \inf_{h \in B} d_{X}(g,h) , \sup_{h \in B} \inf_{g \in A} d_{X}(g,h) \big) 
\end{equation}
If $X$ is a compact space, then the set of closed subsets of $X$ is compact under this Hausdorff metric \citep{henrikson1999completeness}. This also means that it is totally bounded, i.e., for all $\epsilon > 0$ there are a finite number of subsets $A$ with open balls of radius $\epsilon$ in the Hausdorff metric that cover the set of compact subsets. \\
\\ 
All Lie groups $\mathcal{G}$ can be given a metric $d_\mathcal{G}(g, h)$, given by the length of the shortest geodesic from $g$ to $h$ under any chosen Riemannian metric tensor\footnote{ When $\mathcal{G}$ is isomorphic to the Cartesian product of a compact group and $\RR^k$ for some $k$, then there is a natural choice given by the unique bi-invariant metric, see \cite{milnor1976curvatures}, Lemma 7.5. } on $\mathcal{G}$. Since $\mathcal{G}$ is locally compact groups we can take a compact neighbourhood $U$ of the identity $e \in \mathcal{G}$. We define a finite metric $d_{\mathrm{Haus}(U)} : K( \mathcal{G} ) \times K( \mathcal{G} ) \rightarrow [0, \infty) $ given by:
\begin{equation}
    d_{\mathrm{Haus}(U)}( G, H ) = d_{ \mathrm{Haus} }( G \cap U , H \cap U )
\end{equation}
Note that since $\mathcal{G}$ is Hausdorff we know $G \cap U$ is compact for all $G \in K( \mathcal{G} )$, so this is well defined. Moreover, it inherits all the metric properties of $d_{\mathrm{Haus}}$ with base space $U$, for which the set of closed subsets,  $\{ \overline{A} \subseteq U \}$, is compact. The subgroups $K(\mathcal{G})$ can be identified with the subsets $\{ G \cap U \subseteq U \} \subseteq \{ \overline{A} \subseteq U \}$, which means that it is totally bounded (as a subset of a totally bounded space). \\
\\
Now we need to consider how a small distance $d_{\mathrm{Haus}(U)}(G, H)$ allows us to understand the behaviour of $S_{\rho}$ (for some distribution $\rho$ on $G$ with $\mathrm{supp}(\rho) \subseteq U \cap G$) over the space of $H$ invariant functions. This allows us to bound the asymptotic bias when considering $G$-invariant estimators for an $H$ invariant regression function $f$. The main requirements are bounds on derivatives of $f$ and of the action of $\mathcal{G}$. In particular, we say that $\mathcal{G}$ has a $L_\mathcal{G}$-\textbf{Lipshitz} action on $\mathcal{X}$ if $d_\mathcal{X}( g \cdot x, h \cdot x ) \leq L_\mathcal{G} d_\mathcal{G}( g, h) $ for all $x \in \mathcal{X}$ and $g, h \in \mathcal{G}$ (with $d_\mathcal{G}(g, h)$ the metric on the group $\mathcal{G}$ we use to compute Hausdorff distances). Since the action of $\mathcal{G}$ is smooth this condition only means that its derivatives are bounded. An example of this is the following.

\begin{eg}[Example of Lipschitz Action]
    \label{eg:lipschitz_action}
    Consider $\mathcal{G} = SO(3)$ acting naturally on the closed unit ball $\overline{B_{\RR^3}( 0 ,1 )} \subseteq \RR^3$. The Riemannian distance on $\mathcal{G}$ is given by $d_{\mathcal{G}}( g, h ) = 2^{-1/2} \| \log ( g^{-1} h ) \|_F$, which is equal to the (minimal magnitude) angle of the rotation $g^{-1}h$ (Example 4.7 \citet{fletcher2010terse}). Thus we have:
    \begin{align}
        d( g \cdot x, h \cdot x) &= \| g \cdot x - h \cdot x \|_2 = \| (I - g^{-1}h) x \|_2 
    \end{align}
    Thinking of $g, h \in \mathcal{G}$ as matrices here. This expression is maximised when $\| x \|_2$ attains its maximum of $1$, and this chord length is bounded by the angle between $x$ and $g^{-1}h \cdot x$, so we have:
    \begin{align}
        d( g \cdot x, h \cdot x) = \| (I - g^{-1}h) x \|_2 \leq 2^{-1/2} \| \log( g^{-1} h ) \|_2 = d_{\mathcal{G}} (g , h ) 
    \end{align}
    Therefore the action on this space is $1$-Lipschitz. 
\end{eg}

\begin{prop}
    \label{prop:close_groups}
    Let $\rho_G^x$ be distributions on the groups $G$ each with support contained in a compact neighbourhood of the identity $U$, and which can be distinct for each point $x \in \mathcal{X}$. Suppose that $\mathcal{G}$ has an $L_\mathcal{G}$-Lipschitz action on $\mathcal{X}$. Then for all $f \in \mathcal{F}( L, \beta ) \cap L^2_H( \mathcal{X} )$, we have:
    \begin{align}
        \big|( S_{\rho_G^x} f )(x) - f(x) \big| &\leq L L^{\alpha}_\mathcal{G} d_{\mathrm{Haus}(U)} ( G, H )^{ \alpha}
    \end{align}
    for all $x \in \mathcal{X}$, where $\alpha = \min( \beta, 1 )$.
\end{prop}

\subsection{Defining the Error Minimising Symmetries}
\label{ssec:error_min_syms}
Our goal now is to construct a finite set of closed subgroups for which the minimiser of $\hat{\mathcal{E}}( S_{\rho(G)}f_{n} )$ over this set is rate optimal for estimation over the class $\mathcal{F}(L, \beta)$.  Note that for convenience we suppress the scripts on $\rho(G) = \rho^x_m(G)$. To ensure that we capture examples for each possible rate, we first stratify $K(\mathcal{G})$ by the dimension of the regular orbits; i.e., we define:
\begin{equation}
    K_\ell( \mathcal{G} ) = \big\{ G \in K( \mathcal{G} ) : \dim [x]_G = \ell, G \text{ is connected} \big\}.
\end{equation}

\begin{minipage}{\textwidth}
This gives the chain of inclusions that simplify the search space to just the union on the right: 
\begin{equation}
    \bigg\{\begin{array}{c} \text{ Subsets } \\ \text{ of } \mathcal{G}  \end{array}\bigg\} \supseteq  \bigg\{\begin{array}{c} \text{ Subgroups } \\ \text{ of } \mathcal{G}  \end{array}\bigg\} \supseteq  \bigg\{\begin{array}{c} \text{ Closed } \\ \text{ Subgroups of } \mathcal{G}  \end{array}\bigg\} \supseteq  \bigg\{\begin{array}{c} \text{ Closed Connected } \\ \text{ Subgroups of } \mathcal{G}  \end{array}\bigg\} = \bigcup_{\ell = 1}^{\dim [x]_\mathcal{G}} K_\ell ( \mathcal{G} )
\end{equation}
where $[x]_\mathcal{G}$ is any principle orbit of $\mathcal{G}$. In many cases these inclusions are proper, as shown in the following example. 
\begin{eg}
    \label{eg:inclusions_in_SO3}
    Consider the case of $\mathcal{G} = SO(3)$. The subsets that are not subgroups include many sets that do not contain compositions, for example the set of rotations of angle less than $\pi / 4$. There are many subgroups that are not closed, such as the group of rotations generated by a rotation of angle 1 radian around any axis $u \in S^2$. This group is dense in the closed group $S^1_u$ but is not itself closed. Some closed subgroups are not connected, including all finite subgroups with more than one element, and the groups isomorphic to $O(2)$. Lastly the stratification of the last set is simple: there is only one group with $\dim [x]_G = 0$, the trivial group $I$; there is only one with $\dim [x]_G = 2$, $SO(3)$ itself; and the only closed connected subgroups with $\dim[x]_G = 1$ are the groups $S^1_u$ for some axis $u \in S^2$. 
\end{eg}
\end{minipage}
\vspace{0.3cm}

Since $K(\mathcal{G})$ is totally bounded under the metric $d_{\mathrm{Haus}(U)}$ each stratum $K_\ell( \mathcal{G} )$ is totally bounded too (as every subset of a totally bounded set is totally bounded, see lemma \ref{lem:total_bounded_subsets}). So take any $\delta > 0$. For each $\ell$, pick a finite set $\{ G_i^\ell \}_{i = 1}^{k_\ell} $ such that for all $G\in K_\ell( \mathcal{G})$, there exists a $G_i^\ell$ with $d_U(G ,G_i) \leq \delta$. We then take our estimate, the \textbf{(Global) Error Minimising Symmetry} of $G_{\max}(f, \mathcal{G})$ as:
\begin{equation}
    \hat{G}_\delta = \underset{G \in \cup_\ell \{ G_i^\ell \} }{ \mathrm{argmin} } \hat{\mathcal{E}} ( S_{\rho(G)} f_{n} ) = \underset{G \in \cup_\ell \{ G_i^\ell \} }{ \mathrm{argmin} } \sum_{ i = 1}^n ( Y_i' - S_{\rho^{X_i'}_m(G) } f_n( X_i' ) )^2 
\end{equation}
Where the empirical error is taken from an iid copy of the dataset $\mathcal{D}$ used in each $f_n$, given by $\mathcal{D}' = \{ (X_i', Y_i') \}_{ i = 1 }^n$. In the random design context this could be from a split of the data into two iid pieces.

\begin{eg}[Construction of $\delta$-cover of $K( SO(3) )$]
    \label{eg:delta_cover_so3}
    Since $dim [x]_{SO(3)} = 2$, there are only three strata to cover. The easiest is $K_2( SO(3) ) = \{ SO(3) \}$, which is covered finitely by itself. The subgroups in $K_0 ( SO(3) )$  are all finite, and the only connected finite group is the trivial group $I$. Thus we only need to really consider the stratum $K_1( SO(3) )$ which contains only subgroups isomorphic to the circle $S^1$, acting rotationally around the axes $u$ in the unit sphere $S^2$, so $K_1( SO(3) ) = \{ S^1_u : u \in S^2 \}$.    \\
\\
    Consider the distance between subgroups $S^1_u$ and $S^2_v$. For all $g \in S^1_u$ with angle $\phi \in [0, \pi)$, set $h_g \in S^1_v$ as the rotation around $v$ with the same angle of rotation. Then consideration of the real part of the quaternions representing these rotations, we can find $d_{SO(3)}( g, h_g ) = 2^{-1/2} \| \log (g^{-1} h_g ) \|_F = \theta$ where $\theta$ is given by:
\begin{equation}
    \cos(\theta / 2) = \cos(\phi/2)^2 + \langle u, v \rangle \sin(\phi/2)^2 = 1 - (1 - \langle u, v \rangle) \sin( \phi/2 )^2 \geq \langle u, v \rangle 
\end{equation}
when $\langle u, v \rangle \geq 0$. This implies that $\theta \leq 2 \arccos( \langle u, v \rangle )$ in this case, i.e., less than twice the angle between $u$ and $v$. Thus we have:
\begin{equation}
    \sup_{g \in S^1_u} \inf_{h \in S^1_v} d_{SO(3)}( g, h ) \leq \sup_{g \in S^1_u} d_{SO(3)} (g, h_g) \leq 2 \arccos( \langle u, v \rangle) 
\end{equation}
Hence the $\delta$-grid of $K_1( SO(3) )$ is easily given by $\{ S_{u_{ij}}^1 \}$ where $\{ u_{ij} \}$ forms a $2\delta / \pi$ grid of the unit sphere $S^2$ under the angular metric. These $u_{ij}$ can in turn be generated from a $\delta / \pi$ grid of angles in spherical coordinates $(\phi_i, \theta_i) \in [0, 2\pi) \times [0, \pi]$, with $u_{ij} = ( \sin \theta \cos \phi, \sin \theta \sin \phi, \cos \phi )$. This is depicted in figure \ref{fig:stratification_of_K_SO3}. 
\end{eg}

\begin{figure}[h!]
    \centering
    \begin{tikzpicture}
        \draw (0,0) -- (8,0) -- (8,6) -- (0,6) -- (0,0);
        \draw (8,6) node[above left]{Subsets of $SO(3)$};
        \draw (4, 3) ellipse (3 and 2);
        \draw (4, 5) node[above]{Closed Connected Subgroups};
        \draw (1.4, 4) -- (6.6, 4); 
        \draw (1.4, 2) -- (6.6, 2); 
        \draw (6.1,4.5) node[right]{$K_2( SO(3) )$};
        \draw(7,3) node[right]{$K_1( SO(3) )$};
        \draw (6.1,1.5) node[right]{$K_0( SO(3) )$};
        \draw[fill] (4,4.5) circle (0.07cm) node[left]{$SO(3)$};
        \draw[fill] (6,2.5) circle (0.07cm) node[left]{$S^1_{u_1}$};
        \draw[fill] (3,3.5) circle (0.07cm) node[left]{$S^1_{u_2}$};
        \draw[fill] (2,2.3) circle (0.07cm) node[left]{$S^1_{u_3}$};
        \draw[fill] (3.7,2.6) circle (0.07cm) node[right]{$S^1_{u_4}$};
        \draw[fill] (5,3.3) circle (0.07cm) node[left]{$S^1_{u_5}$};
        \draw[fill] (4,1.5) circle (0.07cm) node[left]{$I$};
        \draw[|-|, dashed] (3.1,3.4) -- (2.1,2.2);
        \draw (2.5, 2.9) node[below right]{$\leq \delta$};
    \end{tikzpicture}
    \caption{Example of the search space of the error minimising symmetry when $\mathcal{G} = SO(3)$, continuing from examples \ref{eg:inclusions_in_SO3} and \ref{eg:delta_cover_so3}. As $\delta$ decreases, the middle stratum gains more points in the $\delta$-cover.   }
    \label{fig:stratification_of_K_SO3}
\end{figure}
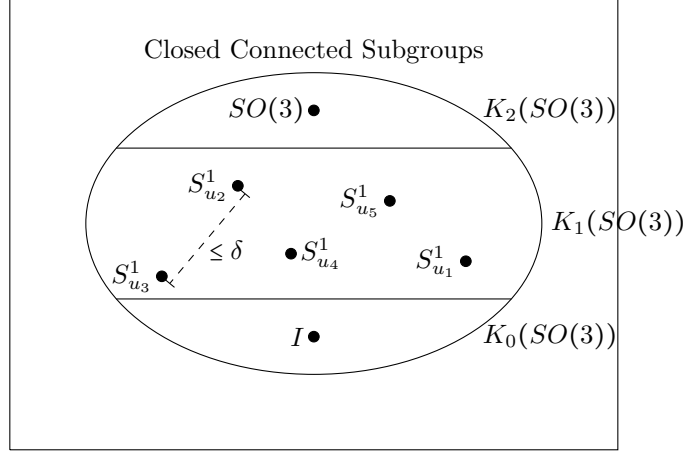

\subsection{Global Rates for the Best Symmetric Estimator }

\label{ssec:global_sym_rates}

We now can study the properties of the \textbf{Best Symmetric Estimator} of $f$: 
\begin{equation}
    \hat{f} = S_{ \rho( \hat{G}_{\delta_n} ) } f_n 
\end{equation}
where $f_n$ is any estimator satisfying assumption set (E), the partial symmetrisation operators are constructed as per section \ref{ssec:par_sym_ests}, $\delta_n$ is a deterministically chosen hyperparameter, and $\hat{G}_{\delta_n}$ is the global error minimising symmetry as per section \ref{ssec:error_min_syms}. We first consider the case where $\mathcal{X}$ is compact, as with $[0,1]^d$, $S^d$, and $\mathbb{T}^d$. In this case the Best Symmetric Estimator achieves the desired rates of adaption to the maximal symmetry of $f$. 

\begin{thm}
    \label{thm:emp_minimiser_adaptive_rates}
    Let $\mathcal{X}$ be a compact space. Let $\delta_n = L_{\mathcal{G}}^{-1} \big( \frac{ \phi_n^{\beta, d - d^\mathcal{G} } }{2 L^2} \big)^{1/ 2 \min( \beta , 1 )} $. Let $\{ G_i^\ell \}_{i = 1}^{k_\ell}$ be a finite $\delta$-cover of the totally bounded sets $K_\ell( \mathcal{G} ) \subseteq K( \mathcal{G} )$ for $\ell = 0, \dots, \dim [x]_\mathcal{G}$. Let $\hat{G}_\delta$ be the minimiser of $\hat{\mathcal{E}}( f_{n,G} )$ over $\cup_\ell \{ G_i^D \}$. Then under assumptions (E) and (Q), the Best Symmetric Estimator $S_{ \rho (\hat{G}_{\delta_n})} f_n$ achieves the rate:
    \begin{equation}
        \sup_{f \in \mathcal{F}( L, \beta)} \frac{ \EE_f\big( \| S_{ \rho (\hat{G}_{\delta_n} )}  f_n - f \|_2^2 \big) }{ \phi_n^{\beta, d - d^{G_{\max}(f, \mathcal{G})} } } \leq 2 ( 1 + \sup_{G \in K(\mathcal{G}) } C_G ) 
    \end{equation}

\end{thm}

Under many circumstances it is possible to bound $C_G$ uniformly over the subgroups in $K(\mathcal{G})$ as in the example below. For this to be true, it is sufficient that $c_\star(x) = c$ for all $x \in \mathcal{X}$ and $\EE( ( R_X^{[X]_{G}} )^{- d^G} ) < A$ for all $G \in K( \mathcal{G} )$.

\begin{eg}  
    Consider the case of $\mathcal{G} = SO(3)$ acting naturally on the unit sphere $S^2$ with a uniform $\mu = U(S^2)$. As shown in example \ref{eg:R_x} we know that $\EE( ( R_x^{[x]_{SO(3)}} )^{- \dim [x]_{SO(3)}} ) = 2$. The identity group has this expectation equal to $1$ because $R$ is constant. The closed connected subgroups with dimension 1 orbits are the groups $S^1_u$ each isomorphic to $SO(2)$. For regular points $x$ we have $R_x^{S^1_u} = \sqrt{ 1 - \langle x, u \rangle }$ because the orbits are circles of this radius around $u$. The inner product has the same distribution as the first coordinate of $X$ by rotational symmetry, which is uniform on $[-1,1]$ (a special case for $U(S^2)$). Thus the distribution of $R_X^{S^1_u}$ is triangular, with $\PP( \sqrt{ 1 - \langle x, u \rangle } < a ) = a^2 / 2$ for $a \in [0,\sqrt{2}]$ and thus has a linear density function over this region. The reciprocal expectation is thus simply
    \begin{equation}
        \EE\big( (R_X^{S^1_u} )^{-1} \big) = \int_{0}^{\sqrt{2}} a^{-1} f_{R_X^{S^1_u}}(a) \dd a = \int_{0}^{\sqrt{2}} \dd a = \sqrt{2}   
    \end{equation}
    This constant is greater than for the larger action because it has more singularities (i.e., the line through $u$ rather than just the origin).
\end{eg}

\subsection{Local Rates for the Best Locally Symmetric Estimator }

In the case that $\mathcal{X}$ is not compact, we can still establish a similar local integrated rates. Instead of minimising the empirical risk, we minimise over a local estimate of the integrated error of $S_{\rho_{\hat{G}_\delta}} f_n$ over a compact set $\Omega = \overline{\Lambda} \subseteq \mathcal{X}$ for a non-empty open set $\Lambda$. This is slightly more technical than the global case, but only because we have to account for the ``no data'' case. Specifically, we define the \textbf{Local Error Minimising Symmetry} as:
\begin{align}
    \label{eq:local_emp_min}
    \hat{G}_\delta(\Omega) = \begin{cases}
        \underset{ G \in \cup_\ell \{ G_i^\ell\} }{ \mathrm{argmin} }
             \hat{\mathcal{E}}_\Omega( S_{\rho(G)} f_n )   & \text{if } \sum_{i = 1}^n \mathbf{1}_{X_i' \in \Omega } > 0  \\
            I &\text{otherwise}
    \end{cases}
\end{align}
for every $x \in \mathcal{X}$, where:
\begin{align}
    \hat{\mathcal{E}}_\Omega ( S_{\rho(G)} f_n ) = \begin{cases}
        \frac{ \sum_{i = 1}^n \big( Y_i' - S_{\rho(G)} f_n (X_i') \big)^2  \mathbf{1}_{X_i' \in \Omega } }{ \sum_{i = 1}^n \mathbf{1}_{X_i' \in \Omega } } &\text{if}  \sum_{i = 1}^n \mathbf{1}_{X_i' \in \Omega } > 0 \\
        L & \text{ otherwise }
    \end{cases} 
\end{align}
and all other terms are as for the global error minimising symmetry. The conditions on our covariate design and bandwidth ensure that the number of terms in $\Omega$ is expected to grow with $n \mu(\Omega)$, so with high probability (i.e., exponentially increasing in $n$) we will have an estimate in the first case of equation (\ref{eq:local_emp_min}). When $\Omega = \overline{ B_\mathcal{X}(x,h) }$ this is effectively a Nadaraya-Watson estimator of the pointwise risk. This then leads to the \textbf{Best Locally Symmetric Estimator}, $S_{\rho( \hat{G}_\delta ( \Omega)  )} f_n$ which obtains rates in the following theorem.

\begin{thm}
    \label{thm:local_rates_LEMS}
    Let $\Omega$ be a compact subset that is the closure of an open subset of $\mathcal{X}$. Let 
    \begin{equation}
        \delta = L_{\mathcal{G}}^{-1} \bigg( \frac{ \phi_n^{\beta, d - d^\mathcal{G} } }{2 L^2} \bigg)^{1/ 2 \min( \beta , 1 )}
    \end{equation}
    Let $\{ G_i^\ell \}_{i = 1}^{k_\ell}$ be a finite $\delta$-cover of the totally bounded sets $K_\ell( \mathcal{G} ) \subseteq K( \mathcal{G} )$ for $\ell = 0, \dots, \dim [x]_\mathcal{G}$. Let $\hat{G}_\delta$ be the local error minimising symmetry. Under assumptions (E) and (Q), we have that  
    \begin{equation}
        \sup_{f \in \mathcal{F}(L, \beta) } \frac{  \EE\big( ( S_{\rho(\hat{G}_\delta (\Omega) )} f_n (X) - f(X) )^2 \mid X \in \Omega \big) }{ \phi^{\beta, d - d^{G_{\max}(f, \mathcal{G} )}}_n } \leq C_0 
    \end{equation}    
    where the constant $C_0$ is given by:
    \begin{equation}
        C_0 =  \sup_{G \in K(\mathcal{G} )} 2 C_{G, \Omega} + 2 + (2 C_{G, \Omega} + 4 L )\tfrac{ 2\beta}{\mu( \Omega) ( 2 \beta + d ) } 
    \end{equation}
    and $C_{G, \Omega}$ is as defined in section \ref{ssec:integrated_risk}. 
\end{thm}

\section{Finite Sample Performance}

\label{sec:experiments}

We now examine the finite sample performance of best symmetric estimators using synthetic data in the contexts of the main examples of this paper, namely 3D Rotations and translations. All code can be found on GitHub at \href{https://github.com/lchristie/M_estimators_of_symmetries}{\cblu https://github.com/lchristie/M\_estimators\_of\_symmetries}.

\subsection{3D Rotations}

\label{ssec:so3_experiments}

Suppose that $\mathcal{X} = \overline{B_{\RR^3}(0, 1)}$, the closed unit ball in $\RR^3$. Let $\mathcal{G} = SO(3)$ (a well known compact Lie group) act naturally by matrix multiplication on $\mathcal{X}$ as in example \ref{eg:3D_rots}. Suppose that $\mu$ is the uniform measure on $\mathcal{X}$, so that $c_x$ is the reciprocal of the Lebesgue volume of $\mathcal{X}$ for all $x$. To estimate the global error minimising symmetry we need:
\begin{enumerate}
    \item To show that this action is $L_{SO(3)}$-Lipschitz;
    \item To construct a $\delta$-cover of each $K_\ell( SO(3) )$.
\end{enumerate}
For this group and covariate space these are done in the examples \ref{eg:lipschitz_action} and \ref{eg:delta_cover_so3} in the previous sections.\\ 
\\
We have used a local constant estimator (LCE), defined explicitly in \ref{app:lce_properties}, originally proposed in \cite{nadaraya1964estimating, watson1964smooth}), with a rectangular kernel as our $f_n$. We consider the holder class $\mathcal{F}( L, \beta = 1)$ and simulated draws of datasets with $Y_i = f(X_i) + \epsilon_i$ where $\epsilon_i \iid N(0, \sigma^2)$ and $X_i \iid U( \mathcal{X} )$, for sample sizes $n \in \{ 30, 50, 75, 100, 150, 200, 300 \} $. We sample two independent copies of the dataset for each simulation, one used to estimate $f$ and the other used to estimate $G_{\max}(f, \mathcal{G})$ for the symmetrised estimator. The baseline LCE uses the union of these datasets to estimate $f$. We have used a bandwidth $h_n^G = n^{-1 / (2\beta + d - d^G) }$ for each partially symmetrised estimator depending on $G \leq \mathcal{G} = SO(3)$, and the minimax optimal $h = h_n^I$ for the baseline LCE. When sampling from any subgroup of $\mathcal{G}$ we use the uniform distribution, as per the results in section \ref{sssec:compact_groups_comps}. We have used 3 choices of $f$, with varying levels of symmetry:

\begin{center}
\begin{tabular}{l|l}
    Regression function & Maximal Symmetry \\ \hline \hline
     $f_1(X) = \cos( \| X \|_2 )$ & $SO(3)$ \\
     $f_2(X) = \cos( \sqrt{ X^2_2 + X^3_2 } )$ & $S^1_x$ \\
     $f_3(X) = X_1^2 + X_2 - 0.6 X_3$ & $I$
\end{tabular} 
\end{center}

We compared an estimated risk of the baseline LCE, $\tfrac{1}{K} \sum_{i = 1}^{K} (f_n(X_i') -f(X_i')  )^2$ for $K = 200$, against the symmetrised LCE, $\tfrac{1}{K} \sum_{i = 1}^{K} ( S_{\hat{G}} f_n (X_i') - f(X_i') )^2$ for $X_i' \iid \mu$ are independent of both the data used to the iid copy used estimate $f_n$ and $\hat{G}_\delta$. The average results of 30 trials for each sample size are  plotted on a log-log scale in figure \ref{fig:so3_simulations}. \\

\begin{figure}[h]
    \centering
    \begin{tabular}{ccc}
         \includegraphics[scale=0.3]{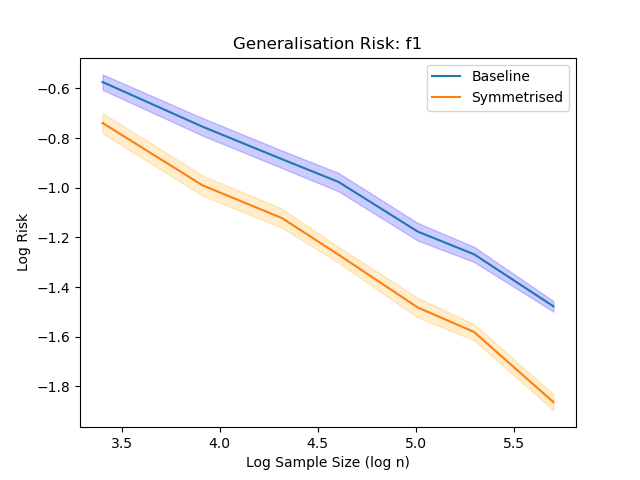} & 
         \includegraphics[scale=0.3]{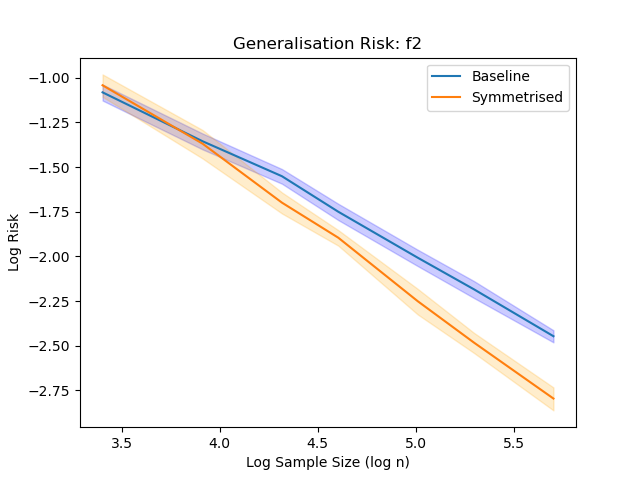} & 
         \includegraphics[scale=0.3]{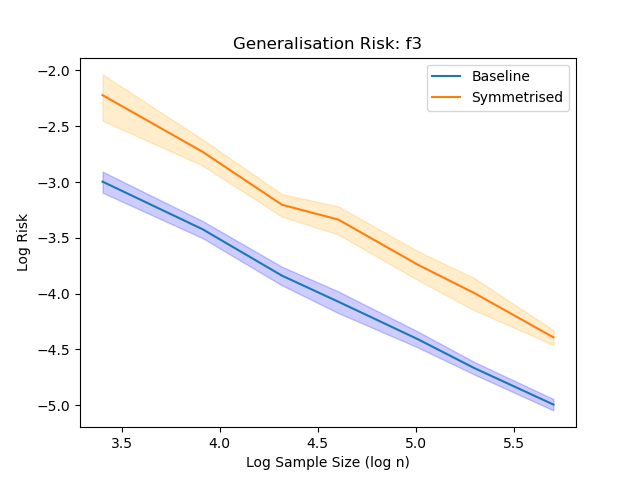} \\
        (a) $f_1$ & (b) $f_2$ & (c) $f_3$
    \end{tabular}
    \caption{Estimated risk for a baseline Local Constant Estimator $f_n$ in blue against the symmetrised $S_{\rho_{\hat{G}_\delta}} f_n$ in orange, as well as $95\%$ Wald confidence intervals for the true risk. Plots are on a log-log scale. }
    \label{fig:so3_simulations}
\end{figure}

We see that the true regression function has some level of symmetry, then our estimator does indeed obtain a faster rate than the baseline LCE. When there is no symmetry within the search group $\mathcal{G}$ our estimator performs slightly worse because of the variance in the estimated group and the half sized sample for for estimation of $f$, but does obtain the same rate of risk decay. Note that the slopes slightly increase with $n$ as the boundary effects diminish.

\subsection{Translational Symmetry}

Consider the situation of example \ref{eg:compact_sparsity}, where $\mathcal{X} =  ( \RR / \ZZ )^2$ is the (flat) 2-torus and $\mathcal{G} = ( \RR / \ZZ )^2$ acts by translation. Recall that if $f$ is invariant to subgroups of translation along an axis then $f$ does not depend on this coordinate of the covariate vector $X$, so captures sparsity in the covariates. As in the previous example, we first establish the required Lipschitz bound and $\delta$-grid.
\begin{enumerate}
    \item The action is $1$-Lipschitz because addition is Lipschitz;
    \item A $\delta$-grid can be constructed by considering the angles of lines through the origin. All closed one dimensional subgroups have this form with a rational angle $\theta$ measured from the $x$-axis, and the distance between them (under the natural metric on $\mathcal{G}$), is bounded as:
    \begin{equation}
        d_\mathcal{G} ( G, H ) \leq \sqrt{2} \sin( \theta_G - \theta_H ) \leq \sqrt{2} ( \theta_G - \theta_H )
    \end{equation}
    Thus a $\delta$-grid is given by lines of angle $\theta_i$ where $\theta_i = 360^\circ i / I$, $I = \lceil 360 / ( \sqrt{2} \delta ) \rceil$ and $i \in \{ 0, 1, \dots,  I  \}$. In practice we would like lines with low denominators so take the union of these grids over $J \leq I$, i.e.:
    \begin{equation}
        \{ G^1_j \}_{j = 1}^k = \bigcup_{J = 1}^I \{ S^1_{ 360^\circ i / J } : i = 0, \dots, J \} 
    \end{equation}
    where $S^1_{\theta}$ is the 1 dimensional group given by the line through the origin in $\mathbb{T}^2$ with angle $\theta$. 
\end{enumerate}
We have used the same base estimator as in the previous example, and simulated draws in the same regression model for three choice of regression functions $f$:

\begin{center}
\begin{tabular}{l|l}
    Regression function & Maximal Symmetry \\ \hline \hline
     $f_1(X) = 1 $ & $\mathcal{G} = \mathbb{T}^2$ \\
     $f_2(X) = \sin( 2 \pi X_1 )$ & $G_2 = S^1_{\theta = 0}$ \\
     $f_3(X) = \cos( 2 \pi ( X_1 - X_2 ) )  $ & $G_3 = S^1_{\theta = -\tfrac{\pi}{2} }  $
\end{tabular} 
\end{center}

The average estimated risks across 100 simulations are plotted in figure \ref{fig:T2_simulations}, where we see that the symmetrised estimator outperforms the baseline LCE in terms of rate for all scenarios, and so has lower absolute performance for $n \geq 150$. At very low sample sizes the performance of the Best Symmetric Estimator is reduced by the variability in the estimated group. 

\begin{figure}[h]
    \centering
    \begin{tabular}{ccc}
         \includegraphics[scale=0.3]{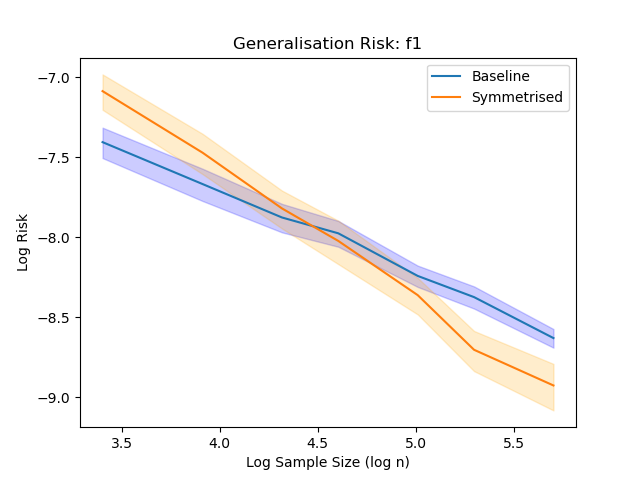} & 
         \includegraphics[scale=0.3]{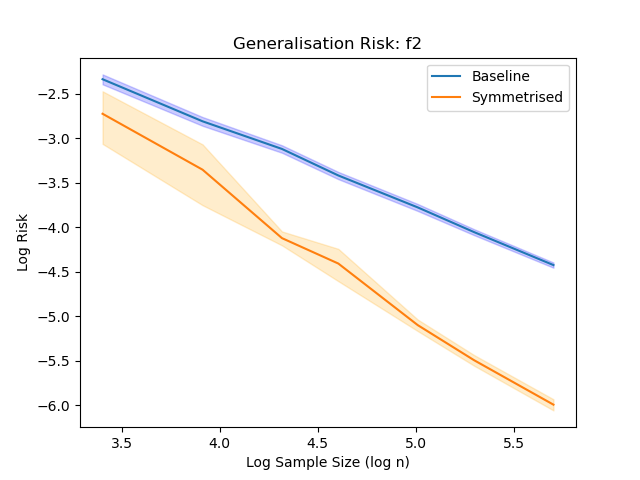} & 
         \includegraphics[scale=0.3]{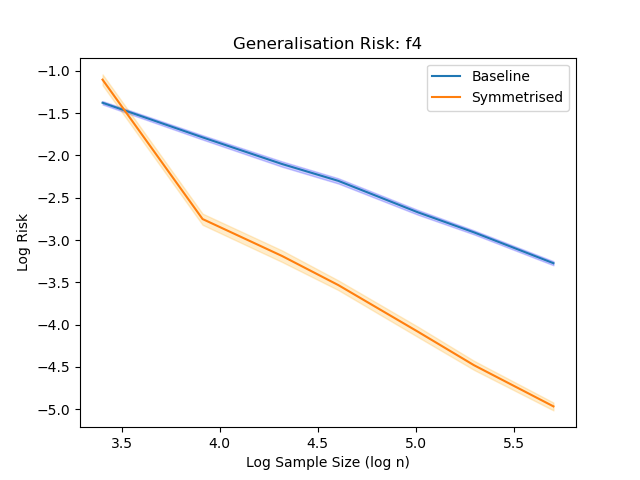} \\
        (a) $f_1$ & (b) $f_2$ & (c) $f_3$
    \end{tabular}
    \caption{Estimated generalisation risk for a baseline Local Constant Estimator $f_n$ in blue against the symmetrised $S_{\rho_{\hat{G}_\delta}} f_n$ when the same data is used in both steps of the $M$-estimation, as well as $95\%$ Wald confidence intervals for the true risk. Plots are on a log-log scale. }
    \label{fig:T2_simulations}
\end{figure}

\section{Discussion}

In this paper we have given sufficient conditions for the adaption to the symmetries present in a non-parametric regression function $f$. This means that we can learn patterns present in our non-parametric regression functions and use these patterns to better generalise. We now mention a few avenues for future work, and make a few remarks.

\subsection{Local Symmetries}

Whilst the requirement that $f$ has a global compact symmetry may be common, it is quite restrictive. A much more general condition is that $f$ obeys a \textit{local} symmetry; where the level set of $f(x)$ and the orbit $[x]_G$ intersect on a compact connected subset of the orbit containing $x$. This captures situations such as digit classification, where we cannot rotate a 6 arbitrarily before it becomes a 9 but we can make small rotational perturbations. The theorems in this paper, even relating to the local error minimising symmetry, require that the local orbits are the same everywhere, which means that they must be able to be patched together into a global symmetry. It would be interesting to understand how the best locally symmetric estimator performs when the requirement of global symmetry is dropped.

\subsection{Improvements on the constants}

One problem with the method presented is that it can be wasteful with the split between data used to estimate $f$ and used to estimate $G_{\max}(f, \mathcal{G})$, which means we see worse finite sample performance when $f$ has no symmetries. Alternatively, a leave one out cross-validation approach could also be examined. \\
\\
Secondly we have ignored non connected subgroups of $\mathcal{G}$ in our search for $G_{\max}(f, \mathcal{G})$. Such groups allow for further reduction in variance and potentially increased performance, at the cost of a much larger search for the error minimising symmetry.

\subsection{Lower Bounds on the estimation of maximal symmetries}

A natural question that arises in this paper is ``does $\EE( d_{\mathrm{Haus}(U)}( \hat{G}, G_{\max}(f, \mathcal{G} ) ) )$ converge uniformly to 0 for all $f \in \mathcal{F}(L, \beta)$, and if so, how quickly?''. Unfortunately, this question has a negative answer because we can have a sequence of non-$G$-invariant functions that approximate a $G$-invariant $f$ arbitrarily well in $L^2(\mathcal{X})$, and distinguishing the symmetries is at least as statistically challenging as distinguishing the functions. This is explicit in the following example.

\begin{eg}
Suppose that $\mathcal{X} = \overline{ B_{\RR^3}( 0, 1 ) }$ and $\mu = U( \mathcal{X})$. Consider the functions $f_3(x) = \exp( - ( x_1^2 + x^2_2) ) $ and $f_2 = \exp( - ( x_1^2 + x^3_2) )$, which are invariant to rotations around the $z$-axis ($S^1_z$) and $y$-axis ($S^1_y$) respectively, as well as reflections through the $x$-$y$ plane ($R_1$) and $x$-$z$ plane ($R_2$) respectively, giving their maximal symmetries. Both are in the H\"{o}lder class $\mathcal{F}(L, 2)$ for some $L$. This is also true for any multiple of these functions by $c \in (0,1]$. \\
\\
Suppose we have normally distributed noise of standard deviation 1, so  $Y_i \mid X_i \sim N( f(X_i), 1 )$. Consider that we can create hypotheses $P_0: f = c f_3$  and $P_1 : f = cf_2$ with $KL( P_0, P_1 )$ arbitrarily small Kullback-Leibler divergence.
\begin{align}
KL(P_0, P_1) &= \int_{\mathcal{X}} \int_{\RR} \log \big( \tfrac{p_0(x,y)}{p_1(x,y)} \big) p_0(x,y) \dd y \mu( \dd x ) \\
			 &= \int_{\mathcal{X}} \int_\RR \log \big( \tfrac{\phi( y - c f_x(x)) }{\phi( y - c f_y(x))} \big) \phi( y - c f_x(x) ) \dd y \mu( \dd x ) \\
			 &= \int_{\mathcal{X}} KL( N(c f_x(x), 1), N( c f_y(x), 1) ) \mu( \dd x ) \\
			 &= \tfrac{c^2}{2} \| f_x - f_y \|_2^2
\end{align}
using the fact that $KL( N(\mu_1, 1) , N(\mu_2, 1) ) = (\mu_1 - \mu_2)^2 / 2$ and writing $\phi$ for the density of a standard normal. So taking $c = \sqrt{  n^{-1} \| f_x - f_y \|_2^{-2} \log(4) } $ ensures $KL( P_0, P_1 ) \leq n^{-1} \log 2 $, and thus Le Cam's lemma \citep{tsybakov2008introduction} implies 
\begin{align}
	\inf_{\hat{G}} \sup_{f \in \mathcal{F}(L, 2) } \EE_f \big( d_{\mathrm{Haus}(SO(3))}( \hat{G}, G_{\max}(f) ) \big) \geq \frac{d_{\mathrm{Haus}(SO(3))}( G_{\max}( f_0, SO(3) ), G_{\max}( f_1, SO(3) ) ) }{16}
\end{align} 
Since this bound is a positive constant, we see this problem is essentially intractable.
\end{eg}

Fortunately, even if the estimated group is not exactly maximal it is usually useful in the estimation of $f$, with sufficiently low asymptotic bias relative to the speed of convergence.

\subsection{A Remark on Equivariance}
In this paper we have considered only invariant functions, for which $f(g \cdot x) = f(x)$. In the machine learning literature there is often consideration of the class of \textbf{equivariant} functions, for which 
\begin{equation}
    f( g \cdot X) = g \star f(X) 
\end{equation}
for some action $\star$ on the domain of $Y$ \citep{bronstein2021geometric}. However, in terms of statistical theory, these functions behave very similarly to invariant functions. This is because the function $f$ can be decomposed into a $G$-invariant function $\tilde{f} : \mathcal{X} / G \rightarrow \mathcal{Y} / G$ with $\tilde{f}(x) = [f(x)]_G$ (i.e., a function which picks the output's orbit) and another $G$-invariant function $r : \mathcal{X} / G \rightarrow G $ (that picks `starting' points) such that $f(x) = r(x) \star \iota_\star( \tilde{f}(x) )$. Both pieces are easier to learn that the original function, both because the input space is reduced to $\mathcal{X} / G$ and the output space is split into $\mathcal{Y} / G$ and $G$. However, because we are still estimating both pieces the errors will still combine - obscuring any additional benefit from the quotient on $\mathcal{Y}$.

\section{Proofs}
\label{sec:proofs}

\subsection{Proofs in Section 3}

\begin{proof}[Proof of Proposition \ref{prop:packing_nos}]
    Each side length can afford at least $\lceil R_x (2h)^{-1} \rceil$ points with $2h$ spacing, and we can do this in each of the dimensions of $T_x \mathcal{M}$. This gives the first inequality in:
    \begin{equation}
    m = | \{ a_i \}_{i = 1}^N | \geq \lceil R_x^{[x]_U} (2h)^{-1} \rceil^{\dim [x]_U }  \geq \lceil ( R_x^{[x]_U} (2h)^{-1} )^{\dim [x]_U } \rceil \geq \max( 1, ( R_x^{[x]_U} (2h)^{-1} )^{\dim [x]_U } )
    \end{equation}
    The second then follows as a basic fact about ceiling functions, as if $\{ x \} = x - \lfloor x \rfloor = 0$ this is trivial and otherwise:
    \begin{align}
        \lceil x^n \rceil &= \Big\lceil \sum_{k = 0}^n \binom{n}{k} \times  \lfloor x \rfloor^k \times \{ x \}^{n - k} \Big\rceil  \leq  \sum_{k = 0}^n \binom{n}{k} \times \lfloor x \rfloor^k \times  \lceil \{ x \}^{n - k} \rceil \leq  \sum_{k = 0}^n \binom{n}{k} \times \lfloor x \rfloor^k \times 1 = \lceil x \rceil^n
    \end{align}
    Now the  points $\{u_i \}$ have $d_\mathcal{X}( u_i, u_j ) \geq 2h$. To see this, let $n_i$ be the normal vector for which $u_i = a_i + n_i$ and $\langle n_i , a_i - x \rangle  = 0$ for each $i$ in the embedding space $\RR^q$. Then we have (with $d$ the metric on the embedding space $\RR^q$),
    \begin{equation}
        d(u_i, u_j)^2 = d( a_i, a_j )^2 + d( n_i, n_j )^2 \geq d(a_i, a_j)^2 \geq (2h)^2
    \end{equation}
    by the orthogonality conditions. Then the isometry of the embedding of $\mathcal{X}$ into $\RR^q$ ensures that the minimal path lengths in $\mathcal{X}$ connecting $u_i$ and $u_j$ are at least as long as the euclidean paths in the ambient space, giving the result. 
\end{proof}

We now provide the proof regarding the pointwise risk of partially symmetrised estimators, which is best analysed by considering the squared bias and the variance of the estimator separately. We will use the shorthand $p_{x,h} = \mu( B_{\mathcal{X}}(x, h) )$, and $g^{i} = g^{i}_{x,h}$. 

\begin{Lem}[Point-wise Bias of $S_{\rho_m^x} f_n$]
        \label{lem:pointwise_bias_disc_sym}
        Under the conditions of proposition \ref{prop:point_wise_error_bound}, we have:
        \begin{equation}
            \big| \EE( S_{\rho_m^x} f_n(x)  ) - S_{\rho_m^x} f (x) \big| \leq B h^\beta + \sup_{i \in [m]} \exp( -n p_{g^i \cdot x,h} ) \| f \|_\infty
        \end{equation}
        when $h = a n^{-1 / ( 2\beta + d - d^G ) }$, for all $n$ in the random design regime. In the fixed deisgn regime these bounds hold for sufficiently large $n$.
    \end{Lem}

    \begin{proof}[Proof of Lemma \ref{lem:pointwise_bias_disc_sym}]
        Using the definition of the symmetrisation operator, we see:
        \begin{align}
            \big| \EE( S_{\rho_m^x} f_n(x)  ) - S_{\rho_m^x} f (x) \big| &= \Big| \EE\Big( \tfrac{1}{m} \sum_{i = 1}^m f_n( g^i \cdot x) - f( g^i \cdot x) \Big) \Big| \\
            &= \Big| \tfrac{1}{m} \sum_{i = 1}^m  \EE\big( f_n( g^i \cdot x) - f( g^i \cdot x) \big) \Big| \\
            &\leq \tfrac{1}{m} \sum_{i = 1}^m \big| \EE\big(  f_n( g^i \cdot x) - f( g^i \cdot x)  \big) \big| \\
            &= \tfrac{1}{m} \sum_{i = 1}^m \big| \EE\big(  \EE( f_n( g^i \cdot x) \mid \mathcal{D}_X ) - f( g^i \cdot x)  \big) \big| \\
            &\leq \tfrac{1}{m} \sum_{i = 1}^m  \EE\big( \big| \EE( f_n( g^i \cdot x) \mid \mathcal{D}_X ) - f( g^i \cdot x) \big| \big)  \\
            &=  \tfrac{1}{m} \sum_{i = 1}^m  \EE\big( \big| \EE( f_n( g^i \cdot x) \mid \mathcal{D}_X ) - f( g^i \cdot x) \big| \mid N_{g^i \cdot x, h} > 0) \big) \PP( N_{g^i \cdot x, h} > 0) + \\
            &\qquad \EE\big( \big| \EE( f_n( g^i \cdot x) \mid \mathcal{D}_X ) - f( g^i \cdot x) \big| \mid N_{g^i \cdot x, h} = 0) \big) \PP( N_{g^i \cdot x, h} = 0) \\
            &\leq  \tfrac{1}{m} \sum_{i = 1}^m  B h^\beta + |f( g^i \cdot x )| \exp( - n p_{ g^i \cdot x, h } )
        \end{align}
        where $N_{x,h} = | B_{\mathcal{X}}(x,h) \cap \mathcal{D}_X | $. This uses assumption (E3c) and lemma \ref{lem:tail_bounds_no_data}. In this fixed design regime these bounds hold for sufficiently large $n$ as per section \ref{ssec:stat_problem}. 
    \end{proof}

\begin{Lem}[Point-wise Variance of $S_{\rho_m^x} f_n$]
        \label{lem:pointwise_var_disc_sym}
        Under the conditions of proposition \ref{prop:point_wise_error_bound}, we have:
        \begin{equation}
            \mathrm{Var} ( S_{\rho_m^x} f_n(x)  )  \leq \tfrac{1}{m} \sup_{i \in [m] } \big( \tfrac{2 V}{ n p_{g^i \cdot x, h} } + V \exp( - n p_{g^i \cdot x, h} / 8 ) + m B^2 h^{2\beta} \big)
        \end{equation}
        when $h = a n^{-1 / ( 2\beta + d - d^G ) }$. This bound holds for all $n$ in the random deisgn regime, and for all sufficiently large $n$ in the fixed design regime. 
    \end{Lem}

    \begin{proof}[Proof of Lemma \ref{lem:pointwise_var_disc_sym}]
        First consider expanding the variance by the law of total variance:
        \begin{align}
            \mathrm{Var} ( S_{\rho_m^x} f_n(x)  ) &= \EE\big( \mathrm{Var}( S_{\rho_m^x} f_n(x) \mid \mathcal{D}_X ) \big) + \mathrm{Var}\big( \EE( S_{\rho_m^x} f_n(x) \mid \mathcal{D}_X ) \big)
        \end{align}
        We now consider each term. Let $N_{x,h} = | B_{\mathcal{X}}(x,h) \cap \mathcal{D}_X |$.
        \begin{align}
            \EE\big( \mathrm{Var}( S_{\rho_m^x} f_n(x) \mid \mathcal{D}_X ) \big) &= \EE\big( \mathrm{Var}( \tfrac{1}{m} \sum_{i = 1}^m f_n(g^i \cdot x) \mid \mathcal{D}_X ) \big) \\
            &= \EE\big( \tfrac{1}{m^2} \mathrm{Var}(  \sum_{i = 1}^m f_n(g^i \cdot x) \mid \mathcal{D}_X ) \big) \\
            &=  \tfrac{1}{m^2} \sum_{i = 1}^m \EE\big( \mathrm{Var}(   f_n(g^i \cdot x) \mid \mathcal{D}_X ) \big)
        \end{align}
        because of the strict locality of $f_n$, the construction of $\{ g^i \}_{i = 1}^m$ which ensures $2h$-spacing. Now we expand each expectation over the $\mathcal{D}_X$-events $\Xi_i = \{ N_{g^i \cdot x, h} = 0 \}$. 
        \begin{align}
            \EE\big( \mathrm{Var}(   f_n(g^i \cdot x) \mid \mathcal{D}_X ) \big) &= \EE( \mathrm{Var}(   f_n(g^i \cdot x) \mid \mathcal{D}_X ) \mid \Xi_i ) \PP( \Xi_i ) + \EE( \mathrm{Var}(   f_n(g^i \cdot x) \mid \mathcal{D}_X ) \mid \Xi_i^C ) (1 -  \PP( \Xi_i ) ) \\
            &= 0 \times \PP( \Xi_i ) + \EE\big( \tfrac{V}{N_{g^i \cdot x, h }} \mid N_{g^i \cdot x , h } > 0 \big) (1 -  \PP( \Xi_i ) ) \\
            &\leq \EE\big( \tfrac{V}{N_{g^i \cdot x, h }} \mid N_{g^i \cdot x , h } > 0 \big) \\
            &= \EE\big( \tfrac{V}{N_{g^i \cdot x, h }} \mid 0 < N_{g^i \cdot x , h } < n p_{g^i \cdot x, h} / 2 )  \big) \PP( 0 < N_{g^i \cdot x , h } < n p_{g^i \cdot x, h} / 2 ) ) + \\
            &\qquad \EE\big( \tfrac{V}{N_{g^i \cdot x, h }} \mid N_{g^i \cdot x , h } > n p_{g^i \cdot x, h} / 2 )  \big) \PP( N_{g^i \cdot x , h } > n p_{g^i \cdot x, h} / 2 ) ) \\
            &\leq V \PP( N_{g^i \cdot x , h } < n p_{g^i \cdot x, h} / 2 ) ) + \tfrac{2 V}{ n p_{g^i \cdot x, h} }  \\
            &\leq \tfrac{2 V}{ n p_{g^i \cdot x, h} } + V \exp( - n p_{g^i \cdot x, h} / 8 )
        \end{align}
        which uses the concentration inequality in Lemma \ref{lem:tail_bounds_no_data}. Thus we have:
        \begin{align}
            \EE\big( \mathrm{Var}( S_{\rho_m^x} f_n(x) \mid \mathcal{D}_X ) \big) &\leq \tfrac{1}{m} \sup_{i \in [m] } \big( \tfrac{2 V}{ n p_{g^i \cdot x, h} } + V \exp( - n p_{g^i \cdot x, h} / 8 ) \big)
        \end{align}
       
        For the second term, consider:
        \begin{align}
            \mathrm{Var}\big( \EE( S_{\rho_m^x} f_n(x) \mid \mathcal{D}_X ) \big) &= \mathrm{Var}\big( \EE( \tfrac{1}{m} \sum_{i = 1}^m f_n(g^i \cdot x) \mid \mathcal{D}_X ) \big) \\
            &= \mathrm{Var}\big(  \tfrac{1}{m} \sum_{i = 1}^m  \EE(f_n(g^i \cdot x) \mid \mathcal{D}_X ) \big) \\
            &\leq \sup_{i \in [m]} \mathrm{Var} \big( \EE(f_n(g^i \cdot x) \mid \mathcal{D}_X ) \big)
        \end{align}

        For ease, let $T_i$ be the random variable $\EE(f_n(g^i \cdot x) \mid \mathcal{D}_X )$. Note that if $N_{g^i\cdot x, h} >0$ then $T_i \in [f(x) - B(x)h^\beta, f(x) - B(x)h^\beta ]$ by assumption (E3) and otherwise $T_i = 0$. Thus we can use assumption (E2) and Popoviciu's inequality for bounded random variables to find:
        \begin{align}
            \mathrm{Var}( T_i ) &= \EE( T_i^2 \mid \Xi ) \PP( \Xi_i ) + \EE( T^2_i \mid  \Xi^C_i ) (1 - \PP( \Xi_i )) - ( \EE( T_i \mid \Xi_i ) \PP( \Xi_i ) + \EE( T_i \mid \Xi^C_i ) ( 1 - \PP(\Xi_i) ) )^2 \\
            &= 0 \times \PP( \Xi_i ) + \EE( T^2_i \mid  \Xi^C_i ) (1 - \PP( \Xi_i )) - ( 0 \times \PP( \Xi_i ) + \EE( T_i \mid \Xi^C_i ) ( 1 - \PP(\Xi_i) ) )^2 \\
            &\leq \mathrm{Var} (T_i \mid \Xi_i^C ) \\
            &\leq B^2 h^{2\beta} 
        \end{align}
        Therefore we have the variance bound:
        \begin{align}
            \mathrm{Var} ( S_{\rho_m^x} f_n(x)  ) \leq \tfrac{1}{m} \sup_{i \in [m] } \big( \tfrac{2 V }{ n p_{g^i \cdot x, h} } + V \exp( - n p_{g^i \cdot x, h} / 8 ) + m B^2 h^{2\beta}  \big)
        \end{align}
        as required.
    \end{proof}
    
\begin{proof}[Proof of Proposition \ref{prop:point_wise_error_bound}]
    The pointwise error of the partially symmetrised estimator has: 
    \begin{align}
        \EE \big( ( S_{\rho_m^x} f_n (x) - S_{\rho_m^x} f(x))^2 \big) &= \big( \EE( S_{\rho_m^x} f_n (x) ) - S_{\rho_m^x} f(x) \big)^2 + \mathrm{Var}\big( S_{\rho_m^x} f_n (x) \big) \\
        &\leq B^2 h^{2\beta} + \sup_{i \in [m]} \exp( -n p_{g^i \cdot x,h} ) \| f \|_\infty + \\
            &\qquad \tfrac{1}{m} \sup_{i \in [m] } \big( \tfrac{2 V }{ n p_{g^i \cdot x, h} } + V \exp( - n p_{g^i \cdot x, h} / 8 ) + m B^2 h^{2\beta}  \big) \label{eq:ptw_risk_bound_1}
    \end{align}
    Now with the choice of bandwidth $h = a n^{-1 / (2 \beta + d - d^G )}$, the bound on $m$ from Prop \ref{prop:packing_nos}, and the distribution property $\mu( B_{\mathcal{X}}(x,h) ) \in [c_x h^{d } , C_x h^{d } ]$ we have that:
    \begin{align}
        h^{2\beta} &\leq a^{2 \beta} \phi_n^{\beta, d - d^G }
    \end{align}
    and
    \begin{align}
        ( m n p_{g^i \cdot x, h} )^{-1} &\leq \tfrac{ (2h)^{d^G} }{ ( R_{x}^{[x]_U})^{d^G } n c_{g^i \cdot x} h^{ d } } \\
        &= \tfrac{ 2^{d^G} }{ c_{g^i \cdot x}  ( R_{x}^{[x]_U})^{d^G } } \times \phi_n^{\beta, d - d^G }
    \end{align}
    Thus we can substitute these back into equation (\ref{eq:ptw_risk_bound_1}) to find:
    \begin{align}
        \EE \big( S_{\rho_m^x} f_n (x) - S_{\rho_m^x} f(x))^2 \big) &\leq \sup_{i \in [m] }  \big( 2 B^2  a^{2 \beta} + + \tfrac{ 2^{1 + d^G}V }{ c_{g^i \cdot x}  ( R_{x}^{[x]_U})^{d^G } } \big) \phi_n^{\beta, \dim{\mathcal{X}} - d^G }  + \\
        &\qquad   (\|f \|_\infty +  V ) \exp( - n p_{g^i \cdot x, h} / 8 )  \\
        &\leq \sup_{g \in U}  2 \big( B(g \cdot x)^2  a^{2 \beta}  + \tfrac{ 2^{d^G}V }{ c_{g \cdot x}  ( R_{x}^{[x]_U})^{d^G } } \big) \phi_n^{\beta, \dim{\mathcal{X}} - d^G }  + \\
        &\qquad   (\|f \|_\infty +  V ) \exp\big( - \tfrac{ c_{g \cdot x} a^{d } }{8}  n^{\tfrac{ 2 \beta - d^G }{2 \beta + \dim\mathcal{X} - d^G } }\big) \\
        &\leq \sup_{g \in U} \Big( 2 \big( B^2  a^{2 \beta}  + \tfrac{ 2^{d^G}V }{ c_{g \cdot x}  ( R_{x}^{[x]_U})^{d^G } } \big) +  \\
        &\qquad \qquad ( \| f \|_\infty + V ) \big( \tfrac{ 16 \beta }{c_{g\cdot x} a^{d} ( 2 \beta - d^G ) } \big)^{ - \tfrac{ 2 \beta }{d}} \Big) \phi_n^{\beta, \dim{\mathcal{X}} - d^G }
    \end{align}
    The result follows from noting that $\sup_g c_{g\cdot x}^{-1} = 1 / \inf_{g} c_{g \cdot x} $ and that $\|f \|_\infty \leq L$ for all $f \in \mathcal{F} ( L, \beta )$.
\end{proof}

\begin{proof}[Proof of Theorem \ref{thm:MISPE_tilde_f}]
    See that the pointwise bound of Proposition \ref{prop:point_wise_error_bound} can be used with the conditional expectation tower law to find: 
    \begin{align}
        \EE(  (S_{\rho_m} f_n (X) - S_{\rho_m} f(X))^2 \mid X \in \Omega) &= \EE( \EE( (S_{\rho_m^X} f_n(X) - S_{\rho_m^X} f(X))^2 \mid X ) \mid X \in \Omega ) \\
            &\leq \EE( C_{G,X} \mid X \in \Omega ) \phi_n^{\beta, d - d^G }
    \end{align}
    Lastly we check the integrability of $C_{G,X}$. See that:
    \begin{align}
        \EE( C_{G,X} \mid X \in \Omega) &= \EE \Big( 2 \big( B^2  a^{2 \beta}  + \tfrac{ 2^{d^G}V }{ c_{*}(X)  ( R_{ X}^{[x]_U})^{d^G } } \big) + ( L + V ) \big( \tfrac{c_{*}(X) a^{d} ( 2 \beta - d^G ) }{ 16 \beta } \big)^{ \tfrac{ 2 \beta }{d}} \Big\mid X \in \Omega \Big) \\
        &=  2 B^2  a^{2 \beta}  + 2 \EE \big( \tfrac{ 2^{d^G}V }{ c  ( R_{ X}^{[x]_U})^{d^G } } \big) + ( L + V ) \EE \Big( \big( \tfrac{c_{*}(X) a^{d} ( 2 \beta - d^G ) }{ 16 \beta } \big)^{ \tfrac{ 2 \beta }{d}}  \Big\mid X \in \Omega \Big) \\
        &\leq 2 B^2  a^{2 \beta}  + 2 \big( \tfrac{ 2^{d^G}V }{ c  } \big) \EE( ( R_{ X}^{[x]_U})^{-d^G } \mid X \in \Omega ) + ( L + V ) \Big( \big( \tfrac{ c a^{d} ( 2 \beta - d^G ) }{ 16 \beta } \big)^{ \tfrac{ 2 \beta }{d}}  \Big)
    \end{align}
    Assumption (Q2) and it's conditioned hypothesis in this theorem then control the expectation term, giving the integrability of $C_X \mid X \in \Omega$.
\end{proof}

\begin{proof}[Proof of Corollary \ref{cor:MSIPE_of_syms}]
	First, note that using the shift invariance of the Haar measure we have that if $g \sim U(G)$, then $g g^i \overset{D}{=} g \sim U( G)$ for all $g^i$, so 
	\begin{align}
		S_G S_{\rho_m^x} \phi ( y )  = m^{-1} \sum_{i = 1}^m \EE( \phi ( (g g^i) \cdot y ) \mid \mathcal{D} ) = m^{-1} \sum_{i = 1}^m \EE( \phi ( g  \cdot y ) \mid \mathcal{D} ) = S_G \phi( y )
	\end{align}
 for all $\phi \in L^2(\mathcal{X} )$. 
	Thus we can say that almost surely: 
	\begin{align}
		\| S_G f_n - S_G f \|_2 &= \| S_G S_{\rho_m} f_n - S_G S_{\rho_m} f \|_2 \leq \| S_{\rho_m} f_n  - S_{\rho_m} f \|_2
	\end{align}
	using the fact that $S_G$ is a projection with operator norm $1$ in the last inequality. Lastly, we can then take expectations and use Proposition \ref{thm:MISPE_tilde_f} to get the stated result. 
\end{proof}

\begin{proof}[Proof of Corollary \ref{cor:MSIPE_f_hat_G}]
The estimator $\hat{f}_{n,G}$ is a Monte-Carlo estimate of $S_G f_n$, so we have:
\begin{align}
    \label{eq:hat_est_bound}
   \EE \| \hat{f}_{n,G } - S_G f \|_2^2 \leq 2 \EE \| \hat{f}_{n,G } - S_G f_n \|_2^2 + 2 \EE \|  S_G f_n - S_G f \|_2^2
\end{align}
using the sum of squares inequality $(a + b)^2 \leq 2 a^2 + 2 b^2$ (Lemma \ref{lem:squared_sum}). The first term describes the Monte-Carlo error of the approximation $\hat{f}_{n,G} \approx S_G f_n$, which is bounded:
\begin{align}
    \EE \| \hat{f}_{n,G } - S_G f_n \|_2^2 &= \EE \big( (  \hat{f}_{n,G }(X) - S_G f_n(X) )^2 \big) \\
    &= \EE\big( \EE( ( \hat{f}_{n,G}(X) - S_G f_n (X) )^2 \mid \mathcal{D}, X ) \big) \\
    &= \EE\big( \mathrm{Var}( \hat{f}_{n,G}(X) \mid \mathcal{D}, X) \big) \\
    &= \tfrac{1}{M} \EE\big( \mathrm{Var}( {f}_{n}(g \cdot X) \mid \mathcal{D}, X) \big) \\
    &= \tfrac{1}{M} \EE\big( ( {f}_{n}(g \cdot X) - S_G f_n (X) )^2 \big) \\
    &\leq \tfrac{3}{M}  \EE\big( \| {f}_{n} \circ (g \cdot)  -  f \circ (g \cdot) \|_2^2 + \| S_G f_n - S_G f \|_2^2 + \| ( S_G f ) \circ (g \cdot)  - f \circ (g \cdot)  \|_2^2 \big) \\
    &\leq \tfrac{6}{M} ( C \phi_n^{\beta, d} + \| f \|_2^2 ) \\
    &\leq \tfrac{6 C + 6 L^2}{M}
\end{align}
using the fact that $S_G f_n = \EE( \hat{f}_{n,G}(X) \mid \mathcal{D}, X )$, the independence of the $g_i$ samples, the sum of squares inequality $(a + b + c)^2 \leq 3( a^2 + b^2 + c^2)$, the $G$-invariance of $\mu$, and lastly the bounds Assumption (E3) and the H\"{o}lder ball bound on $\|f \|_2^2$. Thus we have:
\begin{align}
    \EE ( \| \hat{f}_{n,G } - S_G f \|_2^2 ) &\leq \tfrac{12}{M} (C + L^2 ) + 2 C_G \phi_n^{\beta, d - d^G } \\
    &\leq (12 C + 12 L^2 + 2 C_G ) \phi_n^{\beta, d - d^G }
\end{align}
when $M = n$, as required.
\end{proof}

\subsection{Proof in Section 4}

\begin{proof}[Proof of Proposition \ref{prop:close_groups}]
    Since $f$ is continuous, we know that $f = f_0$ in the definition of invariance. Thus for any such $f \in\mathcal{F}(L, \beta)$, $x \in \mathcal{X}$, and $g \sim \rho_{G}^x$ we have:
    \begin{align}
        ( ( S_{\rho_G^x} f )(x) - f(x) )^2 &= \EE( f(x) - f(g \cdot x) )^2 \\
            &= \EE( \inf_{h \in H \cap U} f(h \cdot x) - f(g \cdot x) )^2 \\
            &\leq L^2 \EE( \inf_{h \in H \cap U} d( h \cdot x, g \cdot x)^\alpha )^2 \\
            &\leq L^2 ( \sup_{g \in G \cap U} \inf_{h \in H \cap U} d( h \cdot x, g \cdot x)^{ \alpha} )^2 \\
            &\leq L^2 \sup_{g \in G \cap U} \inf_{h \in H \cap U} ( L_\mathcal{G} d_\mathcal{G} (g ,h ) )^{2 \alpha} \\
            &\leq L^2 L^{2\alpha}_\mathcal{G} d_{\mathrm{Haus}(U)}( G, H)^{2\alpha} 
    \end{align}
    giving the required inequality.
\end{proof}

\begin{proof}
    Suppose that $\dim [x]_{G_{\max}(f, \mathcal{G}) } = D$. Let $G_i^D$ be any group in $K^D$ with $d_{\mathrm{Haus}}( G_{\max}(f, \mathcal{G}), G_i^D) \leq \delta$ (of which there must be at least one). Then we have:
    \begin{align}
         \| S_{\rho_{\hat{G}_\delta}} f_{n}  -  f \|_2^2 &= \hat{\mathcal{E}}( S_{\rho_{\hat{G}_\delta}} f_{n} ) + ( \| S_{\rho_{\hat{G}_\delta}} f_{n}  -  f \|_2^2 - \hat{\mathcal{E}}( S_{\rho_{\hat{G}_\delta}}  f_{n} ) ) \\
         &\leq \hat{\mathcal{E}}( S_{\rho_{G_i^D}} f_{n} ) + ( \| S_{\rho_{\hat{G}_\delta}}  f_{n}  -  f \|_2^2 - \hat{\mathcal{E}}( S_{\rho_{\hat{G}_\delta}}  f_{n, h} ) ) \\
         &= \| S_{\rho_{G_i^D}}  f_n - f \|_2^2 + (\hat{\mathcal{E}}( S_{\rho_{G_i^D}}  f_{n} ) - \| S_{\rho_{G_i^D}}  f_n - f \|_2^2) + ( \| S_{\rho_{\hat{G}_\delta}} f_{n}  -  f \|_2^2 - \hat{\mathcal{E}}( S_{\rho_{\hat{G}_\delta}} f_{n} ) )
    \end{align}
    Taking expectations (over both the original data used for estimating $f_n$ and the independent copy used for the calculating $\hat{\mathcal{E}}$) then gives:
    \begin{equation}
        R( S_{\hat{G}_\delta} f_{n} ) \leq R( S_{G_i^D} f_{n} ) + \sigma^2 - \sigma^2 
    \end{equation}
    So we need only consider the properties of the group that best approximates $G_{\max}(f, \mathcal{G} )$. We have:
    \begin{align}
        R( S_{G_i^D} f_{n} ) &= \EE( \|  S_{\rho_{G_i^D}}  f_n - f \|_2^2 ) \\
            &\leq 2 \EE( \| S_{\rho_{G_i^D}}  f_n - S_{\rho_{G_i^D}} f \|_2^2 ) + 2 \| S_{\rho_{G_i^D}}  f -  f \|_2^2  \\
            &\leq 2 C_{G_i^D} \phi_{n}^{\beta, d - d^{G_i^D}} + 2 \int_\mathcal{X} ( S_{\rho_{G_i^D}^x}  f(x) -  f(x) )^2 \dd \mu(x) \\
            &\leq 2 C_{G_i^D} \phi_{n}^{\beta, d - d^{G_i^D}} + 2 \int_\mathcal{X} L^2 L_\mathcal{G}^{2\alpha} d_U( G_i^D, G_{\max}(f, \mathcal{G} )^{2\alpha} ) \dd \mu(x) \\
            &\leq 2 C_{G_i^D} \phi_{n}^{\beta, d - d^{G_i^D}} + 2 L^2 L_\mathcal{G}^{2\alpha} \delta^{2\alpha}
    \end{align}
    where $\alpha = \min( \beta, 1 )$. This comes from the sum of squares inequality $(a + b)^2 \leq 2 a^2 + 2 b^2$ (lemma \ref{lem:squared_sum}), Theorem \ref{thm:MISPE_tilde_f}, Proposition \ref{prop:close_groups}, and lastly the $\delta$-closeness of $G_i^D$ to $G_{\max}(f, \mathcal{G})$.   Now we have:
    \begin{align}
        2 L^2 L_{\mathcal{G}}^{2 \min ( \beta ,1 ) } \delta^{2 \min( \beta, 1) } &\leq 2 L^2 L_{\mathcal{G}}^{2 \min ( \beta ,1 ) } ( L_{\mathcal{G}}^{-1} \big( \frac{ \phi_n^{\mathcal{G}} }{2 L^2} \big)^{1/ 2 \min( \beta , 1 )} )^{2 \min ( \beta ,1 ) } \\
        &\leq \phi_n^{\beta, d - d^\mathcal{G} } \\
        &\leq \phi_n^{\beta, d - d^{G_{\max}(f, \mathcal{G})}}
    \end{align}
    Therefore we have:
    \begin{align}
        R( S_{ \rho_{\hat{G}_\delta}} f_n ) &\leq 2 C_G \phi_n^{\beta, d - d^{G_{\max}(f, \mathcal{G})} }  + 2  \phi_n^{\beta, d - d^{G_{\max}(f, \mathcal{G})}} \\
        &\leq 2 ( \sup_{G \in K(\mathcal{G}) } C_G + 1) \phi_n^{\beta, d - d^{G_{\max}(f, \mathcal{G})} }
    \end{align}
    as required.
\end{proof}

\begin{proof}[Proof of Theorem \ref{thm:local_rates_LEMS}]
    Let $G_i^D$ be the minimiser of $d( G_i^D, G_{\max}(f, \mathcal{D} ) )$ over $\cup_\ell \{ G_i^\ell \} $. First, we have the minimisation inequality:
    \begin{align}
        \EE\big( ( S_{\rho_{\hat{G}_\delta (\Omega)}} &f_n (X) - f(X) )^2 \mid f_n,  X \in \Omega \big) \\
        &= \hat{\mathcal{E}}_\Omega(  S_{\rho_{\hat{G}_\delta (\Omega)}} f_n ) + \EE\big( ( S_{\rho_{\hat{G}_\delta (\Omega)}} f_n (X) - f(X) )^2 \mid f_n,  X \in \Omega \big) - \hat{\mathcal{E}}_\Omega(  S_{\rho_{\hat{G}_\delta (\Omega)}} f_n ) \\
        &\leq \hat{\mathcal{E}}_\Omega(  S_{\rho_{G^D_i}} f_n ) + \EE\big( ( S_{\rho_{\hat{G}_\delta (\Omega)}} f_n (X) - f(X) )^2 \mid f_n,  X \in \Omega \big) - \hat{\mathcal{E}}_\Omega(  S_{\rho_{\hat{G}_\delta (\Omega)}} f_n ) \\
        &= \EE( ( S_{\rho_{G^D_i}}f_n(X) - f(X))^2 \mid f_n,  X \in \Omega ) + \\
        &\qquad \hat{\mathcal{E}}_\Omega(  S_{\rho_{G^D_i}} f_n ) - \EE\big( ( S_{\rho_{G^D_i}}f_n(X) - f(X))^2 \mid f_n, X \in \Omega \big) + \\
        &\qquad \EE \big( ( S_{\rho_{\hat{G}_\delta (\Omega)}} f_n (X) - f(X) )^2 \mid f_n, X \in \Omega ) - \hat{\mathcal{E}}_\Omega(  S_{\rho_{\hat{G}_\delta (\Omega)}} f_n )
    \end{align}
    Now conditioned on the event $\Xi'$ that $\Omega$ contains at least one $X_i'$, we have:
    \begin{align}
       \EE(  \hat{\mathcal{E}}_\Omega(  S_{\rho_{G}} f_n ) \mid f_n, \Xi' ) &= \sigma^2 + \EE\big( ( S_{\rho_{G}}f_n(X) - f(X))^2 \mid f_n, X \in \Omega \big) 
    \end{align}
    for all subgroups $G \in \cup_\ell \{ G_i^\ell \}$. In the complementary event $\Xi'^{C}$, we have $\hat{G}_\delta( \Omega) = I$ and so:
    \begin{align}
        \EE(  \hat{\mathcal{E}}_\Omega(  S_{\rho_{G}} f_n ) \mid f_n, \Xi'^{C} ) &= L
    \end{align}
    Thus the integrated error splits over these events as:
    \begin{align}
        \EE\big( ( S_{\rho_{\hat{G}_\delta (\Omega)}} &f_n (X) - f(X) )^2 \mid X \in \Omega \big) \\
            &= \EE\big( \EE\big( ( S_{\rho_{\hat{G}_\delta (\Omega)}} f_n (X) - f(X) )^2 \mid f_n,  X \in \Omega \big) \big) \\
            &= \EE\big( \EE\big( ( S_{\rho_{\hat{G}_\delta (\Omega)}} f_n (X) - f(X) )^2 \mid f_n,  X \in \Omega \big) \mid \Xi_i \big) \PP( \Xi_i ) + \\
            &\qquad \EE\big( \EE\big( ( S_{\rho_{\hat{G}_\delta (\Omega)}} f_n (X) - f(X) )^2 \mid f_n,  X \in \Omega \big) \mid \Xi_i^C \big) \PP( \Xi_i^C ) \\
            &\leq \big(  \EE( ( S_{\rho_{G^D_i}}f_n(X) - f(X))^2 \mid f_n,  X \in \Omega ) + \sigma^2 - \sigma^2 \big) \\
            &\qquad \big( \EE( ( S_{\rho_{\hat{G}_\delta( \Omega )}}f_n(X) - f(X))^2 \mid f_n,  X \in \Omega ) + L - L \big) \PP( \Xi_i'^C ) \\
            &\leq  2 C_{G^D_i, \Omega} \phi_n^{\beta, d - D } + 2 \EE( (S_{\rho_{G^D_i}} f (X) - f( X) )^2 \mid X \in \Omega) + \\
            &\qquad \sup_{G \in K(\mathcal{G})} ( 2 C_{G, \Omega} \phi_n^{\beta, d} +2 \EE( (S_{\rho_{G}} f (x) - f( X) )^2 \mid X \in \Omega) )\PP( \Xi'^C )  \\
            &\leq 2 C_{G^D_i, \Omega} \phi_n^{\beta, d - D } + 2 L^2 L_\mathcal{G}^{2 \alpha} \delta^{2\alpha} +  \sup_{G \in K(\mathcal{G})} ( 2 C_{G, \Omega} \phi_n^{\beta, d} + 4 L )\PP( \Xi'^C ) \\
            &\leq 2 C_{G^D_i, \Omega} \phi_n^{\beta, d - D } + 2 \phi_n^{\beta, d - \dim [x]_\mathcal{G}} + \sup_{G \in K(\mathcal{G})} ( 2 C_{G, \Omega} + 4 L )\PP( \Xi'^C )
    \end{align}
    using the sum of squares inequality (lemma \ref{lem:squared_sum}), the bounds from section \ref{sec:sym_rates}, the fact that the rates $\phi_n^{\beta, d - \ell}$ are faster than $\phi_n^{\beta, d }$ for all $\ell$, the bound of Proposition \ref{prop:close_groups}, and the choice of $\delta$. Lastly note that: 
    \begin{equation}
        \PP( \Xi'^C ) = (1 - \mu( \Omega) )^n \leq \exp( -n \mu( \Omega ) )
    \end{equation}
    thus using lemma \ref{lem:exp_poly_bound} gives the result. 
\end{proof}

\bibliography{main_bib}

\begin{thebibliography}{}

\bibitem[Benton et~al., 2020]{benton2020learning}
Benton, G., Finzi, M., Izmailov, P., and Wilson, A.~G. (2020).
\newblock Learning invariances in neural networks from training data.
\newblock {\em Advances in Neural Information Processing Systems}, 33:17605--17616.

\bibitem[Bietti et~al., 2021]{bietti2021sample}
Bietti, A., Venturi, L., and Bruna, J. (2021).
\newblock On the sample complexity of learning under geometric stability.
\newblock {\em Advances in neural information processing systems}, 34:18673--18684.

\bibitem[Bronstein et~al., 2021]{bronstein2021geometric}
Bronstein, M.~M., Bruna, J., Cohen, T., and Veli{\v{c}}kovi{\'c}, P. (2021).
\newblock Geometric deep learning: Grids, groups, graphs, geodesics, and gauges.
\newblock {\em arXiv preprint arXiv:2104.13478}.

\bibitem[Chiu and Bloem-Reddy, 2023]{chiu2023hypothesis}
Chiu, K. and Bloem-Reddy, B. (2023).
\newblock Hypothesis tests for distributional group symmetry with applications to particle physics.
\newblock In {\em NeurIPS 2023 AI for Science Workshop}.

\bibitem[Christie and Aston, 2023]{christie2023estimating}
Christie, L.~G. and Aston, J.~A. (2023).
\newblock Estimating maximal symmetries of regression functions via subgroup lattices.
\newblock {\em arXiv preprint arXiv:2303.13616}.

\bibitem[Cubuk et~al., 2019]{cubuk2019autoaugment}
Cubuk, E.~D., Zoph, B., Mane, D., Vasudevan, V., and Le, Q.~V. (2019).
\newblock Autoaugment: Learning augmentation strategies from data.
\newblock In {\em Proceedings of the IEEE/CVF Conference on Computer Vision and Pattern Recognition}, pages 113--123.

\bibitem[Dieck, 1987]{dieck1987transformation}
Dieck, T.~T. (1987).
\newblock {\em Transformation groups}.
\newblock de Gruyter.

\bibitem[Elesedy, 2021]{elesedy2021provably}
Elesedy, B. (2021).
\newblock Provably strict generalisation benefit for invariance in kernel methods.
\newblock {\em Advances in Neural Information Processing Systems}, 34:17273--17283.

\bibitem[Elesedy and Zaidi, 2021]{elesedy2021provablyLin}
Elesedy, B. and Zaidi, S. (2021).
\newblock Provably strict generalisation benefit for equivariant models.
\newblock In {\em International Conference on Machine Learning}, pages 2959--2969. PMLR.

\bibitem[Fletcher, 2010]{fletcher2010terse}
Fletcher, T. (2010).
\newblock Terse notes on riemannian geometry.
\newblock \url{http://www.sci.utah.edu/~fletcher/RiemannianGeometryNotes.pdf}.
\newblock Accessed: 2024-02-05.

\bibitem[Garc{\'\i}a-Portugu{\'e}s et~al., 2020]{garcia2020optimal}
Garc{\'\i}a-Portugu{\'e}s, E., Paindaveine, D., and Verdebout, T. (2020).
\newblock On optimal tests for rotational symmetry against new classes of hyperspherical distributions.
\newblock {\em Journal of the American Statistical Association}, 115(532):1873--1887.

\bibitem[Haar, 1933]{haar1933massbegriff}
Haar, A. (1933).
\newblock Der massbegriff in der theorie der kontinuierlichen gruppen.
\newblock {\em Annals of Mathematics}, pages 147--169.

\bibitem[Hawkins, 2012]{hawkins2012emergence}
Hawkins, T. (2012).
\newblock {\em Emergence of the theory of Lie groups: An essay in the history of mathematics 1869--1926}.
\newblock Springer Science \& Business Media.

\bibitem[Henrikson, 1999]{henrikson1999completeness}
Henrikson, J. (1999).
\newblock Completeness and total boundedness of the hausdorff metric.
\newblock {\em MIT Undergraduate Journal of Mathematics}, 1(69-80):10.

\bibitem[Hristache et~al., 2001]{hristache2001structure}
Hristache, M., Juditsky, A., Polzehl, J., and Spokoiny, V. (2001).
\newblock Structure adaptive approach for dimension reduction.
\newblock {\em Annals of Statistics}, pages 1537--1566.

\bibitem[Huang et~al., 2022]{huang2022quantifying}
Huang, K.~H., Orbanz, P., and Austern, M. (2022).
\newblock Quantifying the effects of data augmentation.
\newblock {\em arXiv preprint arXiv:2202.09134}.

\bibitem[Huang and Sen, 2023]{huang2023multivariate}
Huang, Z. and Sen, B. (2023).
\newblock Multivariate symmetry: Distribution-free testing via optimal transport.
\newblock {\em arXiv preprint arXiv:2305.01839}.

\bibitem[Jiang and Tang, 2017]{jiang2017atomic}
Jiang, W. and Tang, L. (2017).
\newblock Atomic cryo-em structures of viruses.
\newblock {\em Current Opinion in Structural Biology}, 46:122--129.

\bibitem[Kondor and Trivedi, 2018]{kondor2018generalization}
Kondor, R. and Trivedi, S. (2018).
\newblock On the generalization of equivariance and convolution in neural networks to the action of compact groups.
\newblock In {\em International Conference on Machine Learning}, pages 2747--2755. PMLR.

\bibitem[Lafferty and Wasserman, 2008]{lafferty2008rodeo}
Lafferty, J. and Wasserman, L. (2008).
\newblock Rodeo: Sparse, greedy nonparametric regression.
\newblock {\em Annals of Statistics}, 36(1):28--63.

\bibitem[Lim et~al., 2019]{lim2019fast}
Lim, S., Kim, I., Kim, T., Kim, C., and Kim, S. (2019).
\newblock Fast autoaugment.
\newblock {\em Advances in Neural Information Processing Systems}, 32.

\bibitem[Lyle et~al., 2020]{lyle2020benefits}
Lyle, C., van~der Wilk, M., Kwiatkowska, M., Gal, Y., and Bloem-Reddy, B. (2020).
\newblock On the benefits of invariance in neural networks.
\newblock {\em arXiv preprint arXiv:2005.00178}.

\bibitem[Mardia and Jupp, 2009]{mardia2009directional}
Mardia, K.~V. and Jupp, P.~E. (2009).
\newblock {\em Directional statistics}, volume 494.
\newblock John Wiley \& Sons.

\bibitem[Marron and Dryden, 2021]{marron2021object}
Marron, J.~S. and Dryden, I.~L. (2021).
\newblock {\em Object oriented data analysis}.
\newblock Chapman and Hall/CRC.

\bibitem[Mei et~al., 2021]{mei2021learning}
Mei, S., Misiakiewicz, T., and Montanari, A. (2021).
\newblock Learning with invariances in random features and kernel models.
\newblock In {\em Conference on Learning Theory}, pages 3351--3418. PMLR.

\bibitem[Milnor, 1976]{milnor1976curvatures}
Milnor, J. (1976).
\newblock Curvatures of left invariant metrics on lie groups.
\newblock {\em Advances in Mathematics}, 21(3):293--329.

\bibitem[Mitzenmacher and Upfal, 2017]{mitzenmacher2017probability}
Mitzenmacher, M. and Upfal, E. (2017).
\newblock {\em Probability and computing: Randomization and probabilistic techniques in algorithms and data analysis}.
\newblock Cambridge university press.

\bibitem[Nadaraya, 1964]{nadaraya1964estimating}
Nadaraya, E.~A. (1964).
\newblock On estimating regression.
\newblock {\em Theory of Probability \& Its Applications}, 9(1):141--142.

\bibitem[Nash, 1956]{nash1956imbedding}
Nash, J. (1956).
\newblock The imbedding problem for riemannian manifolds.
\newblock {\em Annals of Mathematics}, 63(1):20--63.

\bibitem[Rice, 1984]{rice1984bandwidth}
Rice, J. (1984).
\newblock Bandwidth choice for nonparametric regression.
\newblock {\em Annals of Statistics}, pages 1215--1230.

\bibitem[Schmidt-Hieber, 2020]{schmidt2020nonparametric}
Schmidt-Hieber, J. (2020).
\newblock Nonparametric regression using deep neural networks with relu activation function.
\newblock {\em Annals of Statistics}, 48(4):1875--1897.

\bibitem[Stone, 1982]{stone1982optimal}
Stone, C.~J. (1982).
\newblock Optimal global rates of convergence for nonparametric regression.
\newblock {\em Annals of Statistics}, pages 1040--1053.

\bibitem[Tahmasebi and Jegelka, 2023a]{tahmasebi2023exact}
Tahmasebi, B. and Jegelka, S. (2023a).
\newblock The exact sample complexity gain from invariances for kernel regression.
\newblock {\em Advances in Neural Information Processing Systems}, 36.

\bibitem[Tahmasebi and Jegelka, 2023b]{tahmasebi2023sample}
Tahmasebi, B. and Jegelka, S. (2023b).
\newblock Sample complexity bounds for estimating probability divergences under invariances.
\newblock {\em arXiv preprint arXiv:2311.02868}.

\bibitem[Tsybakov, 2008]{tsybakov2008introduction}
Tsybakov, A.~B. (2008).
\newblock {\em Introduction to Nonparametric Estimation}.
\newblock Springer Science \& Business Media.

\bibitem[Van~der Vaart, 2000]{van2000asymptotic}
Van~der Vaart, A.~W. (2000).
\newblock {\em Asymptotic Statistics}, volume~3.
\newblock Cambridge university press.

\bibitem[Watson, 1964]{watson1964smooth}
Watson, G.~S. (1964).
\newblock Smooth regression analysis.
\newblock {\em Sankhy{\=a}: The Indian Journal of Statistics, Series A}, pages 359--372.

\end{thebibliography}
\bibliographystyle{apalike}

\newpage
\appendix

\newpage 
\section{Useful Lemmas}

\begin{Lem}[Sum of Squares Inequality]
    \label{lem:squared_sum}
    For any $a \in \RR^m$, we have:
    \begin{equation}
        \Big( \sum_{i = 1}^m a_i \Big)^2 \leq m \sum_{i = 1}^m a_i^2
    \end{equation}
\end{Lem}
\begin{proof}
    This is a basic application of the Cauchy-Schwarz inequality: consider
    \begin{equation}
        n^{-1} \sum a_i = \sum a_i b_i \leq \big( \sum a_i^2 \big)^{1/2} \big( \sum b_i^2 \big)^{1/2} = \big( \sum a_i^2 \big)^{1/2} \big( n^{-1} \big)^{1/2}
    \end{equation}
    then simply square and rearrange.
\end{proof}

\begin{Lem}
    \label{lem:exp_poly_bound}
    For any positive constants $A, B, a, \alpha, \beta$ we have:
    \begin{equation}
        A n^{-\beta} + B \exp( - a n^{\alpha} ) \leq ( A + B \big( \tfrac{ \beta}{ a \alpha } \big)^{ \tfrac{\beta}{\alpha - 1}} ) n^{-\beta}
    \end{equation}
    if $\alpha \neq 1$, and otherwise:
    \begin{equation}
        A n^{-\beta} + B \exp( - a n^{\alpha} ) \leq ( A + B \tfrac{\beta}{a} ) n^{-\beta}
    \end{equation}
\end{Lem}

\begin{proof}
    First, note that $\exp( - a n^\alpha ) n^\beta$ is maximised when $n = n^* = \big( \tfrac{ \beta}{ a \alpha } \big)^{ \tfrac{1}{\alpha - 1}}$ if $\alpha \neq 1$, so
    \begin{align}
        A n^{-\beta} + B \exp( - a n^{\alpha} )  &= A n^{-\beta} + B  \exp( - a n^{\alpha} ) n^\beta n^{-\beta} \\
            &\leq A n^{-\beta} + B \exp( - a (n^*)^\alpha ) (n^*)^\beta n^{-\beta} \\
            &\leq A n^{-\beta} + B (n^*)^\beta n^{-\beta} \\
            &\leq ( A + B \big( \tfrac{ \beta}{ a \alpha } \big)^{ \tfrac{\beta}{\alpha - 1}} ) n^{-\beta}
    \end{align}
    as required. Similarly if $\alpha = 1$ then the maximum of $\exp( - a n^\alpha ) n^\beta$ occurs at $n = \beta / a$ giving the second bound.
\end{proof}

\begin{Lem}
        \label{lem:tail_bounds_no_data}
        In the random design context, where $X_i \iid \mu$, we have:
        \begin{align}
            \PP( | B_{\mathcal{X}}(x,h) \cap \mathcal{D}_X | = 0 ) &\leq \exp( - n p_{x,h} ) \\
            \PP( | B_{\mathcal{X}}(x,h) \cap \mathcal{D}_X | \leq \tfrac{ n p_{x,h} }{2}   ) &\leq \exp( - n p_{x,h}  / 8 )
        \end{align}
        In the fixed design context these bounds also hold under the conditions in section \ref{ssec:stat_problem} for sufficiently large $n$.
\end{Lem}

\begin{proof}
    We can use the Taylor expansion of $\ln(1 - x)$ around $x = 0$ to bound the first term:
    \begin{align}
            \PP( | B_\mathcal{X}(x,h) \cap D_X | = 0 ) &= \PP( X_i \not\in B_\mathcal{X}(x,h) \forall X_i ) \\
                                &= (1 - p_{x,h} )^n \\
                                &= \exp( n \ln ( 1 - p_{x,h} ) ) \\
                                &= \exp( n ( - \sum_{k = 1}^\infty p_{x,h}^k k^{-1} ) ) \\
                                &\leq \exp( - n p_{x,h} )
    \end{align}

    For the second, we note that $| B_{\mathcal{X}}(x,h) \cap \mathcal{D}_X |  \sim \mathrm{Binom}( n , p_{x,h} )$ and take the Chernoff bound:
    \begin{equation}
        \PP( N \leq  n p / 2 ) \leq \exp( - np / 8 )
    \end{equation}
    for $N \sim \text{Binom}(n, p)$, a proof of which is in \cite{mitzenmacher2017probability}.
\end{proof}

\begin{Lem}
    \label{lem:total_bounded_subsets}
    Let $(X, d_X)$ be a totally bounded metric space. Any sub metric space $(Y, d_Y)$ of $(X, d_X)$, where $Y \subseteq X$ and $d_Y(y_1, y_2) = d_X( y_1, y_2)$ for all $y_1, y_2 \in Y$, is also totally bounded. 
\end{Lem}

\begin{proof}
    Let $\epsilon > 0$ and let $\{ x_i \}_{i = 1}^m$ be such that $X \subseteq \cup_{i = 1}^m B_X (x_i, \epsilon / 2)$, which exists by the definition of total boundedness. Let $Y$ be any subset of $X$ and pick $y_i \in B_X( x_i, \epsilon / 2) \cap Y$ if this set is non-empty and pick any other $y_i \in Y$ otherwise. Then for all $y \in Y$ we have:
    \begin{equation}
        d( y, y_i ) \leq d(y, x_i ) + d(x_i, y_i) \leq \tfrac{ \epsilon }{ 2 } + \tfrac{ \epsilon }{ 2 } 
    \end{equation}
    where $x_i$ is any such that $y \in B_X(x_i, \epsilon / 2) $. Thus $\{ y_i \}_{i = 1}^m$ forms an $\epsilon$-cover of $Y$ as required. 
\end{proof}

\newpage
\section{Local Constant Estimator Satisfies Assumptions (E)}
\label{app:lce_properties}

Recall that the Nadaraya-Watson, or Local Constant Estimator, is given by the simple expression:
\begin{equation}
    f_n(x)  = \begin{cases}
        \frac{ \sum_{i = 1}^n \mathbf{1}_{d(x, X_i) < h} Y_i }{ \sum_{i = 1}^n \mathbf{1}_{d(x, X_i) < h} } & \text{if } \sum_{i = 1}^n \mathbf{1}_{d(x, X_i) < h} > 0 \\
        0 & \text{otherwise}
    \end{cases}  
\end{equation}

for a chosen bandwidth $h$. The choice of default value is taken by some authors as the global mean $\frac{1}{n} \sum_{i = 1}^n Y_i$ but we take it to be a $0$ to satisfy (E2). Since the probability of this occurring decays exponentially with $n$, it does not impact our results. We will show that this estimator satisfies the assumptions in set (E), when estimating over the H\"{o}lder class $\mathcal{F}(L, 1)$, as long as the bandwidth $h$ decays monotonely to $0$ with $n$. These results extend to $\beta > 1$ when this estimator is extended to a local polynomial estimator of degree $\ell = \lfloor \beta \rfloor$ (see \cite{tsybakov2008introduction}).

\subsection{LCE is Strictly Local (E1)}

This is basically a trivial consequence of the local definition of $f_n$. If $x$ and $y$ are $2h$ separated then their values (conditioned on $\mathcal{D}_X$) depend only on the $Y_i$ with the $X_i$ in their local balls. Since these have to be disjoint sets, they are independent. If either of these sets are empty then $f_n$ is deterministic and thus independent of all random variables.

\subsection{LCE is Optimal (E3)}

We consider the other three equations in assumption (E3) in sequence.

\subsubsection{Bias Term}

Since $f \in \mathcal{F}(L, 1)$, we know that:
\begin{equation}
   | \EE( Y_i \mid X_i ) - f(x) | = | f(X_i) - f(x) | \leq B d_\mathcal{X}(x, X_i)^\beta
\end{equation}
Thus conditioned on the event that there is at least one data point in the ball $B_\mathcal{X}(x,h)$, we must have 
\begin{equation}
    \big| \EE( f_n \mid | B(x,h) \cap D_X | > 0) - f(x) \big|= \big| \tfrac{1}{| B(x,h) \cap D_X |} \sum_{ i : X_i \in B_\mathcal{X}(x,h) } \EE( f(X_i) - f(x) \mid X_i \in B(x,h) ) \big| \leq B h^\beta
\end{equation}
Thus the probability that the event (E3) is true is at least the probability that $B_\mathcal{X}(x,h)$ contains any $X_i$. In this fixed design case this is deterministic, and so the design and bandwidth need to be chosen so that this is true. In the case that $\mathcal{X} = [0,1]^d$, this can be done with a uniform grid, or with an $h$-packing of an other compact manifold $\mathcal{X}$. In the random design context, when $X_i \iid \mu$, we use Lemma \ref{lem:tail_bounds_no_data}.

\subsubsection{Variance Term}

In the same context as the above, the variance conditioned on the covariates is given by 
\begin{equation}
    \mathrm{Var}( f_n(x) \mid \mathcal{D}_X ) = \begin{cases}
        \tfrac{\sigma^2}{ | B_\mathcal{X}(x,h) \cap \mathcal{D}_X | } & \text{if } B_\mathcal{X}(x,h) \cap \mathcal{D}_X \text{ is non-empty} \\
        0 & \text{otherwise}
    \end{cases}
\end{equation}
where $N = | B_\mathcal{X}(x,h) \cap \mathcal{D}_X |$ is the number of data-points in the local ball. Thus with $V = 2 \sigma^2$ (which is clearly integrable), we know that this conditional variance is bounded by $V (n \mu( B_{\mathcal{X}}(x,h) ))^{-1}$ whenever $N \geq n\mu( B_{\mathcal{X}}(x,h) ) / 2$. In the fixed design context on a compact covariate space, this can be guaranteed to be true. In the random design context, we have that $N \sim \mathrm{Binom}( n , \mu( B_\mathcal{X}(x,h) ) )$ and so use the Chernoff bound on the lower tail of the binomial (Lemma \ref{lem:tail_bounds_no_data}) to show that the probabilistic variance bound is satisfied in this context.

\subsubsection{Point-wise Error}

The previous two cases cover the usual behaviour of $f_n$, but the exceptional behaviour on the exponentially low probabilities need to be considered to ensure the error doesn't blow up faster than this probability decays. By the same reasoning as in the last two sections we have:
\begin{align}
    \EE( ( f_n(x) - f(x) )^2 \mid \mathcal{D}_X ) \leq \begin{cases}
        B^2 h^{2\beta} + \tfrac{ V }{ n h^{d} } & \textit{if } N > n \mu( B_\mathcal{X}(x, h) ) / 2 > 0 \\
        B^2 h^{2\beta} + \sigma^2 & \textit{if } 0 < N \leq n \mu( B_\mathcal{X}(x, h) ) / 2 \\
        f(x)^2 & \textit{if } N = 0
    \end{cases}
\end{align}
Since $h$ monotonely decreases to $0$, and $\PP( N = 0 )\leq \PP( N \leq n \mu( B_\mathcal{X}(x, h) ) / 2 ) \leq \exp( - n \mu( B_{\mathcal{X}}(x,h) ) / 8)$, we have: 
\begin{align}
    \EE (\EE( ( f_n(x) - f(x) )^2 \mid \mathcal{D}_X ) ) &\leq B^2 h^{2\beta} + \tfrac{ V }{ n h^{d} } +  (B^2 h^{2\beta} + \sigma^2+ f(x)^2 )\exp( - n \mu( B_{\mathcal{X}}(x,h) ) / 8 )  \\
    &\leq C (B^2 h^{2\beta} + \tfrac{ V }{ n h^{d} })
\end{align}
for the positive constant $C$ given by:
\begin{equation}
    C = 1 + \big(B^2  + \sigma^2 + f(x)^2 \big) \sup_{n \in \NN} \frac{ \exp(- n \mu( B_{\mathcal{X}}(x,h) ) / 8 ) }{ B^2 h^{2\beta} + \tfrac{ V }{ n \mu( B_{\mathcal{X}}(x,h) ) } } 
\end{equation}
which we must show is finite. To see this, note that the supremum term is bounded above:
\begin{equation}
    \sup_{n \in \NN} \frac{ \exp(- n \mu( B_{\mathcal{X}}(x,h) ) / 8 ) }{ B^2 h^{2\beta} + \tfrac{ V }{ n \mu( B_{\mathcal{X}}(x,h) ) } } \leq V^{-1}  \sup_{n \in \NN} n \mu( B_{\mathcal{X}}(x,h) ) \exp(- n \mu( B_{\mathcal{X}}(x,h) ) / 8 ) 
\end{equation}
It is easy to check that $x \exp( - x / 8) $ is maximised at $x = 8$, so we know that:
\begin{equation}
    C \leq 1 + 8 e^{-1} V^{-1} \big(B^2  + \sigma^2 + f(x)^2 \big) \leq 1 + \tfrac{8}{2\sigma^2 e} ( B(x) + \sigma^2 + \| f \|_\infty^2 ) < \infty
\end{equation}
Lastly note that since $f \in \mathcal{F}( L , 1)$, $\| f \|_\infty^2 \leq L$ and so $C$ is indeed finite.

\end{document}